\documentclass[12pt,reqno]{amsart}
\usepackage{amsmath, amsfonts, amssymb, amsthm,mathrsfs, cite}
\def\Z{\mathbb Z}

\def\N{\mathbb N}
\def\Q{\mathbb Q}

\def\F{\mathbb F}

\def\sU{{\mathscr U}}\def\sV{{\mathscr V}}

\def\cL{{\mathcal L}}

\def\cU{{\mathcal U}}
\def\cZ{{\mathcal Z}}

\def\fH{{\mathfrak H}}
\def\fL{{\mathfrak L}}

\def\p{\mathfrak p}

\def\bc{{\bf c}}

\def\jacob #1#2{\genfrac{(}{)}{}{}{#1}{#2}}
\def\pmod #1{\ ({\mathrm{mod}}\ #1)}
\def\mod #1{{\mathrm{mod}}\ #1}

\theoremstyle{plain}
\newtheorem{theorem}{Theorem}

\newtheorem{lemma}{Lemma}
\newtheorem{corollary}{Corollary}
\theoremstyle{definition}

\theoremstyle{remark}
\newtheorem*{Rem}{Remark}
\renewcommand{\theequation}{\arabic{section}.\arabic{equation}}
\renewcommand{\thetheorem}{\arabic{section}.\arabic{theorem}}
\renewcommand{\theconjecture}{\arabic{section}.\arabic{conjecture}}
\renewcommand{\theproposition}{\arabic{section}.\arabic{proposition}}
\renewcommand{\thelemma}{\arabic{section}.\arabic{lemma}}
\renewcommand{\thecorollary}{\arabic{section}.\arabic{corollary}}

\begin{document}
\title
{A local-global theorem for $p$-adic supercongruences}

\author{Hao Pan}
\address[Hao Pan]{School of Applied Mathematics, Nanjing University of Finance and Economics, Nanjing 210023, People's Republic of China}
\email{haopan79@zoho.com}

\author{Roberto Tauraso}
\address[Roberto Tauraso]{Dipartimento di Matematica,
Universit\`a di Roma ``Tor Vergata'',
via della Ricerca Scientifica,
00133 Roma, Italy}
\email{tauraso@mat.uniroma2.it}

\author{Chen Wang}
\address[Chen Wang]{Department of Mathematics, Nanjing
University, Nanjing 210093, People's Republic of China}
\email{cwang@smail.nju.edu.cn}



\begin{abstract}
Let $\Z_p$ denote the ring of all $p$-adic integers and call $$\cU=\{(x_1,\ldots,x_n):\,a_1x_1+\ldots+a_nx_n+b=0\}$$ a hyperplane over $\Z_p^n$, where at least one of $a_1,\ldots,a_n$ is not divisible by $p$. We prove that if a sufficiently regular $n$-variable function is zero modulo $p^r$ over some suitable collection of $r$  hyperplanes, then it is zero modulo $p^r$  over the whole $\Z_p^n$. We provide various applications of this general criterion  by establishing several $p$-adic analogues of hypergeometric identities.
\end{abstract}
\maketitle

\section{Introduction}
\setcounter{lemma}{0} \setcounter{theorem}{0}
\setcounter{equation}{0}

Among special functions, the  classical hypergeometric series ${}_{r+1}F_r$ have a long history of investigation in mathematical analysis. They are defined as
\begin{equation}
{}_{r+1}F_r\bigg[\begin{matrix}\alpha_0&\alpha_1&\cdots&\alpha_r\\ &\beta_1&\cdots&\beta_r\end{matrix}\bigg|\,z\bigg]:=\sum_{k=0}^{\infty}\frac{(\alpha_0)_k\cdots(\alpha_r)_k}{(\beta_1)_k\cdots(\beta_r)_k}\cdot\frac{z^k}{k!},
\end{equation}
where
$$
(\alpha)_k:=\begin{cases}
\alpha(\alpha+1)\cdots(\alpha+k-1),&\text{if }k\geq 1,\\
1,&\text{if }k=0,
\end{cases}
$$
is the so-called $k$-th rising factorial of $\alpha$, or Pochhammer symbol.

In the last decade, there is a growing interest in studying the arithmetic properties of truncated hypergeometric series, that we denote  by the following subscript notation
\begin{equation}
{}_{r+1}F_r\bigg[\begin{matrix}\alpha_0&\alpha_1&\cdots&\alpha_r\\ &\beta_1&\cdots&\beta_r\end{matrix}\bigg|\,z\bigg]_{n}:=\sum_{k=0}^{n}\frac{(\alpha_0)_k\cdots(\alpha_r)_k}{(\beta_1)_k\cdots(\beta_r)_k}\cdot\frac{z^k}{k!}.
\end{equation}
In 1996, van Hamme \cite{vanHamme97} gave a list of conjectures concerning $p$-adic analogues of several formulas of Ramanujan. Later, Rodriguez-Villegas \cite{RVillegas03} studying hypergeometric families of Calabi-Yau manifolds, discovered (numerically) a number of other supercongruences. Some of them have been proved in \cite{Mortenson03,Mortenson04} where Mortenson, with the help of the Gross-Koblitz formula,  determined ${}_{2}F_1\bigg[\begin{matrix}\alpha&1-\alpha\\ &1\end{matrix}\bigg|\,1\bigg]_{p-1}$ modulo $p^2$ for $\alpha\in\{1/2,1/3,1/4,1/6\}$. For instance, he showed that for any prime $p\geq 5$,
\begin{equation}
{}_{2}F_1\bigg[\begin{matrix}\frac12&\frac12\\ &1\end{matrix}\bigg|\,1\bigg]_{p-1}\equiv\jacob{-1}{p}\pmod{p^2},
\end{equation}
where $\jacob{\;\cdot\;}{\cdot}$ denotes the Legendre symbol.

Afterwards, Z.-H. Sun \cite{SunZH14} extended Mortenson's result to the general $p$-adic integer $\alpha$. Let $\Z_p$ denote the ring of all $p$-adic integers and $\Z_p^\times:=\{x\in\Z_p:\,p\nmid x\}$. Z.-H. Sun proved that for each odd prime $p$ and $\alpha\in\Z_p^\times$,
\begin{equation}\label{ZHSun}
{}_{2}F_1\bigg[\begin{matrix}\alpha&1-\alpha\\ &1\end{matrix}\bigg|\,1\bigg]_{p-1}\equiv(-1)^{\langle-\alpha\rangle_p}\pmod{p^2},
\end{equation}
where $\langle x\rangle_p$ is the least nonnegative residue of $x$ modulo $p$, i.e., $\langle x\rangle_p\in\{0,1,\ldots,p-1\}$ and $x\equiv \langle x\rangle_p\pmod{p}$. For more results related to supercongruences of truncated hypergeometric series, the readers may refer to
\cite{Ahlgren01, AhOn00,
BaSa20, BD1805,CoHa91,DFLST16,Gu17,Gu20,He17a,He17b,
KaCh19,KLMSY16,Kilbourn06,Liu18,Liu20,Long11,LoRa16,
LTNZ1705,McCarthy12a,McCarthy12b,Mortenson03,Mortenson04,Mortenson05,
Mortenson08,OsSc09,OsStZu1701,OsZu16,PZ15,RR1803,SunZH11,SunZH13,
SunZH14,SunZW09,SunZW11, SunZW12,SunZW13,Swisher15,
Tauraso12,Tauraso18,Tauraso20,Zudilin09}.

Some of them involve the  Morita's  $p$-adic Gamma function $\Gamma_p$ (cf. \cite[Chapter 7]{Robert00}) which is  the $p$-adic  analogue of the classical Gamma function $\Gamma$:  we set $\Gamma_p(0)=1$ and, for any integer $n\geq 1$,
$$
\Gamma_p(n):=(-1)^n\prod_{\substack{1\leq k<n\\ k\not\equiv0\pmod{p}}}k.
$$
Since $\N$ is a dense subset of $\Z_p$ with respect to the  $p$-adic norm $|\cdot|_p$, for each $x\in\Z_p$, we may extend the definition of $\Gamma_p$ as
$$
\Gamma_p(x):=\lim_{\substack{n\in\N\\ |x-n|_p\to 0}}\Gamma_p(n).
$$
It follows that
\begin{equation}\label{Gammapxx1}
\frac{\Gamma_p(x+1)}{\Gamma_p(x)}=\begin{cases}-x,&\text{if }p\nmid x,\\
-1,&\text{if }p\mid x.\end{cases}
\end{equation}
Furthermore, the classical Legendre relation (cf. \cite[p. 371]{Robert00}) has the following counterpart
\begin{equation}\label{Gammapx1x}
\Gamma_p(x)\Gamma_p(1-x)=(-1)^{\langle-x\rangle_p-1}.
\end{equation}

Recently, in \cite{MP17}, Mao and Pan  obtained
many congruences modulo $p^2$ involving truncated hypergeometric series and $p$-adic Gamma functions. For example, they proved that if $\langle-\alpha\rangle_p$ is even, $\langle-\alpha\rangle_p\leq\langle-\beta\rangle_p<(p-\langle-\alpha\rangle_p)/2$ and $(\alpha-\beta+1)_{p-1}$ is not divisible by $p^2$, then
\begin{align}\label{PMDixon}
&{}_3F_2\bigg[\begin{matrix}\alpha&\alpha&\beta\\
&1&\alpha-\beta+1\end{matrix}\bigg|\,1\bigg]_{p-1}\notag\\
&\quad\equiv-\frac{2\Gamma_p(1+\frac12\alpha)\Gamma_p(1+\alpha-\beta)\Gamma_p(1-\frac12\alpha-\beta)}{
\Gamma_p(1+\alpha)\Gamma_p(1-\frac{1}2\alpha)\Gamma_p(1-\beta)\Gamma_p(1+\frac{1}2\alpha-\beta)}\pmod{p^2}.
\end{align}
Clearly \eqref{PMDixon} is a $p$-adic analogue of the following special case of Dixon's well-poised sum formula (cf. \cite[Theorem 3.4.1]{AAR99}):
$$
{}_3F_2\bigg[\begin{matrix}
\alpha&\alpha&\beta\\
&1&\alpha-\beta+1
\end{matrix}\bigg|\,1\bigg]=\frac{\Gamma(1+\frac12\alpha)\Gamma(1+\alpha-\beta)\Gamma(1-\frac12\alpha-\beta)}{\Gamma(\alpha+1)\Gamma(1-\frac12\alpha)\Gamma(1-\beta)\Gamma(1+\frac12\alpha-\beta)}.
$$
It is worth noting that most of the congruences concerning truncated hypergeometric series modulo $p^2$, such as \eqref{PMDixon}, can be  reduced to some complicated congruences modulo $p$ involving the harmonic numbers. Mao and Pan found that those congruences modulo $p$ can be proved in a unified way by considering the derivatives of original hypergeometric series.

Keeping in mind this approach, in this paper we shall investigate the $p$-adic analogues of various classical hypergeometric identities modulo higher powers of $p$. For example, setting $\beta=\alpha$ in \eqref{PMDixon}, we get
\begin{align}\label{PMDixon2}
{}_3F_2\bigg[\begin{matrix}\alpha&\alpha&\alpha\\ &1&1\end{matrix}\bigg|\,1\bigg]_{p-1}\equiv\frac{2\Gamma_p(1+\frac12\alpha)\Gamma_p(1-\frac32\alpha)}{\Gamma_p(1+\alpha)\Gamma_p(1-\alpha)\Gamma_p(1-\frac12\alpha)^2}\pmod{p^2}.
\end{align}
As we shall see later, \eqref{PMDixon2} is also valid for modulo $p^3$, i.e.,
\begin{align}\label{PMDixon2p3}
{}_3F_2\bigg[\begin{matrix}\alpha&\alpha&\alpha\\ &1&1\end{matrix}\bigg|\,1\bigg]_{p-1}\equiv\frac{2\Gamma_p(1+\frac12\alpha)\Gamma_p(1-\frac32\alpha)}{\Gamma_p(1+\alpha)\Gamma_p(1-\alpha)\Gamma_p(1-\frac12\alpha)^2}\pmod{p^3}.
\end{align}
However, in order to obtain a congruence modulo $p^{r+1}$, we should  evaluate the $r$-th derivatives of $p$-adic Gamma functions and such computations become quite complicated while the order increases.
Here to avoid this tedious task, we managed to establish a general result which has some geometric flavour and is interesting in its own right.

Before giving the statement we introduce a few notions. First, we need to consider the Taylor expansions of functions over $\Z_p$ (cf. \cite[Section 5.3]{Robert00}).
Suppose that $0\leq r\leq p-1$ and $f:\,\Z_p\to\Z_p$. We say that $f$ has a {\it strong Taylor expansion} of order $r$ at $x=\alpha$, if there exist $A_1(\alpha),\ldots,A_{r}(\alpha)\in\Z_p$ such that
\begin{equation}\label{Taylorexpansionorderr}
f(\alpha+tp)\equiv f(\alpha)+\sum_{k=1}^r\frac{A_k(\alpha)\cdot (tp)^k}{k!}\pmod{p^{r+1}}
\end{equation}
for any $t\in\Z_p$. As we shall see later, since $r\leq p-1$, the coefficient $A_i(\alpha)$ is uniquely determined by its modulo $p^{r+1-i}$ for each $1\leq i\leq r$. Furthermore, it is not difficult to check that for each $s_0\in\Z_p$, we have the Taylor expansion
$$
f(\alpha+s_0p+tp)\equiv f(\alpha+s_0p)+
\sum_{j=1}^r\frac{A_j(\alpha+s_0p)\cdot (tp)^j}{j!}
\pmod{p^{r+1}},
$$
where
$$
A_j(\alpha+s_0p)=
\sum_{k=0}^{r-j}\frac{A_{k+j}(\alpha)}{k!}\cdot(s_0p)^k.
$$
This means that if $f$ has the strong Taylor expansion at $x=\alpha$, then it has such expansion also at $x=\alpha+sp$ for each $s\in\Z_p$.

Similarly, for a function of several variables $f(x_1,\ldots,x_n)$ over $\Z_p^n$, we say that $f$ has a strong Taylor expansion of order $r$ at the point $(\alpha_1,\ldots,\alpha_n)$, provided that
\begin{align*}
&f(\alpha_1+t_1p,\ldots,\alpha_n+t_np)\\
&\quad\equiv f(\alpha_1,\ldots,\alpha_n)
+\sum_{\substack{k_1,\ldots,k_n\geq 0\\
1\leq k_1+\cdots+k_n\leq r}}\frac{A_{k_1,\ldots,k_n}(\alpha_1,\ldots,\alpha_n)}{k_1!\cdots k_n!}
\prod_{i=1}^n(t_ip)^{k_i}\pmod{p^{r+1}}
\end{align*}
for any $t_1,\ldots,t_n\in\Z_p$, where those $A_{k_1,\ldots,k_n}(\alpha_1,\ldots,\alpha_n)\in\Z_p$.

Next, let us introduce the definition of hyperplane over the  finite field with $p$ elements $\F_p$. For convenience, we identify $\F_p$ as $\{0,1,\ldots,p-1\}$.
Let $\F_p^n$ denote the $n$-dimensional vector space over $\F_p$. Suppose that $$\fL(x_1,\ldots,x_n)=a_1x_1+\cdots+a_nx_n+b$$ is a linear function over $\F_p^n$ with $a_1,\ldots,a_n$ are not all zero.
Let
$$
\sU_{\fL}:=\{(x_1,\ldots,x_n)\in\F_p^n:\,\fL(x_1,\ldots,x_n)=0\}.
$$
We call $\sU_{\fL}$ a {\it hyperplane} over $\F_p^n$. For example, $\sU_{x+y-1}=\{(x,1-x):\,x\in\F_p\}$ is a hyperplane over $\F_p^2$.

Going back to $\Z_p^n$, let $\tau_p$ be the natural homomorphism from $\Z_p$ to $\F_p$, i.e., $$\tau_p(x):=\langle x\rangle_p$$ for each $x\in\Z_p$.  Suppose that $$L(x_1,\ldots,x_n)=a_1x_1+\cdots+a_nx_n+b$$ is a linear function over $\Z_p^n$ with $\tau_p(a_j)\neq 0$ for some $1\leq j\leq n$. Let
$$
\cU_L:=\{(x_1,\ldots,x_n)\in\Z_p^n:\,L(x_1,\ldots,x_n)=0\}
$$
and we write
$$
\tau_p(\cU_L):=\{(\tau_p(x_1),\ldots,\tau_p(x_n)):\,(x_1,\ldots,x_n)\in\cU_L\}.
$$
Clearly $$\tau_p(\cU_L)=\sU_{\tau_pL}$$ forms a hyperplane over $\F_p^n$, where
\begin{equation}
\tau_pL(x_1,\ldots,x_n):=\tau_p(a_1)x_1+\cdots+\tau_p(a_n)x_n+\tau_p(b)
\end{equation}
can be viewed as a linear function over $\F_p^n$.
So we also call $\cU_L$ a {\it hyperplane} over $\Z_p^n$.
For some hyperplanes $\cU_{L_1},\ldots,\cU_{L_m}$, we say $\{\cU_{L_1},\ldots,\cU_{L_m}\}$
is {\it admissible} provided that $$\tau_p(\cU_{L_i})\neq \tau_p(\cU_{L_j})$$ for each $1\leq i<j\leq r$.  For example, for $p=3$, $\{\cU_{x-1},\cU_{y-1},\cU_{x+y-1},\cU_{x-y-1}\}$ is admissible, but $\{\cU_{x+y-1},\cU_{4x-2y-4}\}$ is not.

Below, for convenience, we always use $\sU_\fL$ and $\cU_L$ to represent the hyperplanes over $\F_p^n$
and $\Z_p^n$ respectively.

The main result of this paper is the following theorem.

\begin{theorem}\label{main} Suppose that $n,r\geq 1$ and $p$ is a prime with
\begin{equation}\label{pbinomr12}
p>\binom{r+1}{2}.
\end{equation}
Let $\cU_{L_1},\ldots,\cU_{L_r}$ be some hyperplanes over $\Z_p^n$ such that $\{\cU_{L_1},\ldots,\cU_{L_r}\}$ is admissible.
Assume that the function $\Psi(x_1,\ldots,x_n)$ over $\Z_p^n$ has a strong Taylor expansion of order $r$ at the point $(0,\ldots,0)$, and
$$
\Psi(s_1p,\ldots,s_np)\equiv 0\pmod{p^r}
$$
for each
$$
(s_1,\ldots,s_n)\in\bigcup_{i=1}^r\cU_{L_i}.
$$
Then
\begin{equation}
\Psi(s_1p,\ldots,s_np)\equiv0\pmod{p^{r}}
\end{equation}
for each $(s_1,\ldots,s_n)\in\Z_p^n$.
\end{theorem}

In other words, Theorem \ref{main} says that if a congruence modulo $p^r$ holds over some $r$ hyperplanes of $\Z_p^n$, then it is also valid over the whole set $\Z_p^n$. So the above statement can be read as: if a congruence is true locally, then it is true globally.

A minor disadvantage of Theorem \ref{main} is the requirement \eqref{pbinomr12}. Actually, motivated by Theorem \ref{psiaiprt}, we conjecture that the requirement $p>\binom{r+1}{2}$ might be weakened to $p>r$. Of course we can always verify each congruence for those primes $p\leq\binom{r+1}{2}$ via numerical computations.

Let us briefly describe our strategy to prove congruence \eqref{PMDixon2p3} by using Theorem \ref{main}. First, we construct a function $\psi(x,y,z)$, which is $3$-differentiable over $(p\Z_p)^3$, such that \eqref{PMDixon2p3} is equivalent to
$$
\psi(p\alpha_p^*,p\alpha_p^*,p\alpha_p^*)\equiv 0\pmod{p^3},
$$
where $\alpha_p^*=(\alpha+\langle-\alpha\rangle_p)/p\in\Z_p$ is the so-called dash operation.
Next, by applying Dixon's well-poised sum formula, we can show that
$$\psi(0,sp,tp)=\psi(rp,0,tp)=\psi(rp,sp,0)=0$$
for each $r,s,t\in\Z_p$. Since $\{\cU_x,\cU_y,\cU_z\}$ is admissible over $\Z_p^3$, in view of Theorem \ref{main}, we get the desired result.

The paper is organized as follows. In the next section, we will discuss a preliminary result for one-variable functions.
Then, in the third section, with the help of Schwartz-Zippel lemma, we prove Theorem \ref{main} which involves multi-variable functions.
Since the existence of the Taylor expansion in $\Z_p$ is a crucial tool to ensure the successful application of Theorem \ref{main}, in Section 4 we discuss the Taylor expansions of rational functions and of the $p$-adic Gamma functions. In the remaining sections, we will provide various applications of Theorem \ref{main} by establishing several $p$-adic analogues of hypergeometric identities.

\section{The one-variable case}
\setcounter{equation}{0}\setcounter{theorem}{0}
\setcounter{lemma}{0}\setcounter{corollary}{0}

\begin{theorem}\label{psiaiprt} Let $p$ be a prime.
Assume that $1\leq r\leq p$ and $\psi(x)$ is a function over $\Z_p$ with the following strong Taylor expansion of order $r-1$
at $x=0$:
$$
\psi(tp):=\psi(0)+\sum_{k=1}^{r-1}\frac{A_k}{k!}\cdot(tp)^k\pmod{p^r}.
$$
Let $a_1,\ldots,a_r\in\Z_p$ with $\tau_p(a_i)\neq\tau_p(a_j)$ for any $1\leq i<j\leq r$. Suppose that
$$
\psi(a_ip)\equiv0\pmod{p^{r}}
$$
for each $1\leq i\leq r$.
Then
\begin{equation}\label{psisppr}
\psi(sp)\equiv 0\pmod{p^{r}}
\end{equation}
for every $s\in\Z_p$. Furthermore, for each $1\leq k\leq r-1$, we have
\begin{equation}\label{Ak0prk}
A_k\equiv0\pmod{p^{r-k}}.
\end{equation}
\end{theorem}
\begin{proof}
According to the Taylor expansion of $\psi(x)$,
$$
\begin{cases}
0\equiv\psi(a_1p)\equiv\psi(0)+\frac{A_1}{1!}\cdot a_1p+\frac{A_2}{2!}\cdot a_1^2p^2+\cdots+\frac{A_{r-1}}{(r-1)!}\cdot a_1^{r-1}p^{r-1}\pmod{p^r},\\
0\equiv\psi(a_2p)\equiv\psi(0)+\frac{A_1}{1!}\cdot a_2p+\frac{A_2}{2!}\cdot a_2^2p^2+\cdots+\frac{A_{r-1}}{(r-1)!}\cdot a_2^{r-1}p^{r-1}\pmod{p^r},\\
\quad\vdots\\
0\equiv\psi(a_rp)\equiv\psi(0)+\frac{A_1}{1!}\cdot a_rp+\frac{A_2}{2!}\cdot a_r^2p^2+\cdots+\frac{A_{r-1}}{(r-1)!}\cdot a_r^{r-1}p^{r-1}\pmod{p^r}.
\end{cases}
$$
Since the Vandermonde determinant is
$$
\left|\begin{matrix}
1&a_1&a_1^2&\cdots&a_1^{r-1}\\
1&a_2&a_2^2&\cdots&a_2^{r-1}\\
\vdots&\vdots&\vdots&\ddots&\vdots\\
1&a_r&a_r^2&\cdots&a_r^{r-1}\\
\end{matrix}\right|=\prod_{1\leq i<j\leq r}(a_j-a_i)\not\equiv0\pmod{p},
$$
we find that
\begin{equation}\label{psi0pr}
\psi(0)\equiv0\pmod{p^r}
\end{equation}
and, for each $1\leq k\leq r-1$,
\begin{equation}\label{Akkpkpr}
\frac{A_k}{k!}\cdot p^k\equiv0\pmod{p^r}.
\end{equation}
 Thus \eqref{psisppr} immediately follows from \eqref{psi0pr}, \eqref{Akkpkpr} and the Taylor expansion of $\psi$.
\end{proof}

According to Theorem \ref{psiaiprt}, the $p$-adic Taylor expansion of a function over $\Z_p$ is unique in the following sense.

\begin{corollary} Let $p$ be a prime and $1\leq r<p$. Suppose that  the function $\psi(x)$ over $\Z_p$ has two forms of Taylor expansion at $x=\alpha$:
\begin{align*}
\psi(\alpha+tp)\equiv&\psi(\alpha)+\sum_{k=1}^{r}\frac{A_k}{k!}\cdot(tp)^k\pmod{p^{r+1}}\\
\equiv&\psi(\alpha)+\sum_{k=1}^{r}\frac{B_k}{k!}\cdot(tp)^k\pmod{p^{r+1}}.
\end{align*}
Then for $1\leq k\leq r$,
$$
A_k\equiv B_k\pmod{p^{r+1-k}}.
$$
\end{corollary}
\begin{proof}
Now the zero function has the Taylor expansion
$$
0\equiv 0+\sum_{k=1}^{r}\frac{A_k-B_k}{k!}\cdot(tp)^k\pmod{p^{r+1}}.
$$
So by \eqref{Ak0prk},
$$
A_k-B_k\equiv 0\pmod{p^{r+1-k}}.
$$
\end{proof}
With the help of Theorem \ref{psiaiprt}, we now present a short proof of  \eqref{ZHSun}. Assume that $f(x)$ is a polynomial over $\Z_p$. It is straightforward to verify  that for each $1\leq r\leq p-1$ and $a\in\Z_p$, $f(x)$ has the Taylor expansion
$$
f(a+tp)\equiv f(a)+\sum_{k=1}^{r}\frac{f^{(k)}(a)}{k!}\cdot(tp)^k\pmod{p^{r+1}},
$$
where $f^{(k)}$ denotes the formal derivative of order $k$  of $f(x)$.

Let $a=\langle-\alpha\rangle_p$ and let
\begin{equation}\label{psizhsun}
\psi(x):={}_{2}F_1\bigg[\begin{matrix}-a+x&1+a-x\\ &1\end{matrix}\bigg|\,1\bigg]_{p-1}-(-1)^a.
\end{equation}
Clearly $\psi$ is a polynomial over $\Z_p$, and,
by the Chu-Vandermonde identity (cf. \cite[Corollary 2.2.3]{AAR99}), it is easy to show that
\begin{equation}\label{psizhsun0p}
\psi(0)=\psi(p)=0.
\end{equation}
According to Theorem \ref{psiaiprt} with $a_1=0$ and $a_1=1$, we have
$$
\psi(sp)\equiv0\pmod{p^2}
$$
for each $s\in \Z_p$.
In particular, setting $s=(\alpha+a)/p$, we get \eqref{ZHSun}.

Furthermore, by \eqref{Ak0prk}, for each $s\in\Z_p$ we have
\begin{equation}\label{psidzhsun}
\psi'(sp)\equiv\psi'(0)\equiv0\pmod{p}.
\end{equation}
Note that, for each $k\geq 1$,
\begin{equation}
\frac{d((x)_k)}{dx}=(x)_k\sum_{j=0}^{k-1}\frac{1}{x+j}.
\end{equation}
It follows from \eqref{psidzhsun} that
\begin{align*}
\sum_{k=1}^{p-1}\frac{(\alpha)_k(1-\alpha)_k}{(k!)^2}\sum_{j=0}^{k-1}\bigg(\frac{1}{\alpha+j}-\frac{1}{1-\alpha+j}\bigg)\equiv0\pmod{p}.
\end{align*}
In particular, for $d\in\{3,4,6\}$, and for each prime $p>3$, we have
$$
\sum_{k=1}^{p-1}\frac{(\frac1d)_k(\frac{d-1}{d})_k}{(k!)^2}\cdot H_{dk,\chi_d}\equiv0\pmod{p},
$$
where
$\chi_d$ denotes the unique non-trivial character modulo $d$, and
$$
H_{k,\chi}=\sum_{j=1}^k\frac{\chi(j)}{j}.
$$

\section{The multi-variable case and the proof of Theorem \ref{main}}
\setcounter{equation}{0}\setcounter{theorem}{0}
\setcounter{lemma}{0}\setcounter{corollary}{0}

In order to prove Theorem \ref{main}, we need to reduce the multi-variable function to a one-variable function.

Suppose that $v_1,\ldots,v_n,c_1,\ldots,c_n\in\F_p$ and at least one of $v_1,\ldots,v_n$ is non-zero. Write
$$
l(\vec{v},\bc):=\left\{(v_1t+c_1,\ldots,v_nt+c_n)\,:\,t\in\F_p\right\},
$$
where $\vec{v}=(v_1,\ldots,v_n)$ and $\bc=(c_1,\ldots,c_n)$.
Clearly $l(\vec{v},\bc)$ forms a line over $\F_p^n$.

\begin{lemma}\label{linehyperplane}
Suppose that $\sU_{\fL_1},\ldots,\sU_{\fL_r}$ are distinct hyperplanes  over $\F_p^n$ and let $\bc\in \F_p^n$.  If $p>\binom{r+1}2$ and
$$
\bc\not \in\bigcup_{1\leq i<j\leq r}(\sU_{\fL_i}\cap \sU_{\fL_j}),
$$
then there exists $\vec{v}\in\F_p^n\setminus\{(0,\ldots,0)\}$ such that
$$
|l(\vec{v},\bc)\cap\sU_{\fL_i}|=1
$$
for each $1\leq i\leq r$, and
$$
l(\vec{v},\bc)\cap\sU_{\fL_i}\neq l(\vec{v},\bc)\cap\sU_{\fL_j}
$$
for each $1\leq i<j\leq r$.
\end{lemma}
\begin{proof}
For each $1\leq i\leq r$, let $$\lambda_i(x_1,\ldots,x_n):=\fL_i(x_1,\ldots,x_n)-\fL_i(0,\ldots,0)$$
and
$$
\mu_i:=\fL_i(c_1,\ldots,c_n).
$$
Since $\fL_i$ is linear, it is easy to see that
$$
\fL_i(v_1t+c_1,\ldots,v_nt+c_n)=t\cdot\lambda_i(v_1,\ldots,v_n)+\mu_i.
$$
For $1\leq i<j\leq n$, $\bc\not \in \sU_{\fL_i}\cap\sU_{\fL_j}$ and $\sU_{\fL_i}\not=\sU_{\fL_j}$ imply that
$$\mu_j\cdot\lambda_i(x_1,\ldots,x_n)-
\mu_i\cdot\lambda_j(x_1,\ldots,x_n)$$
is a polynomial of degree $1$.
Let
$$
f(x_1,\ldots,x_n):=\prod_{i=1}^r\lambda_i(x_1,\ldots,x_n)\cdot\prod_{1\leq i<j\leq r}\big(\mu_j\cdot\lambda_i(x_1,\ldots,x_n)-
\mu_i\cdot\lambda_j(x_1,\ldots,x_n)\big).
$$
It is clear that $f$ is a polynomial over $\F_p$ of degree $r+\binom{r}{2}=\binom{r+1}{2}$.

We need the Schwartz-Zippel lemma (cf. \cite[Theorem 6.13]{LN97}) as follows:

\bigskip
\textit{Let $f(x_1,\ldots,x_n)$ be a polynomial over $\F_p$  of degree $d\geq 0$. Then the set of zeros of $f$ in $\F_p^n$, i. e.
$\cZ_f:=\{(a_1,\ldots,a_n)\in \F_p^n:\,\ f(a_1,\ldots,a_n)=0\}$,
contains at most $d\cdot p^{n-1}$ elements.
}
\bigskip

Hence, since $p>\binom{r+1}{2}$, by the Schwartz-Zippel lemma, it follows that
$$|\cZ_f|\leq \binom{r+1}{2}\cdot p^{n-1}<p^n-1$$
and we find that there exists $(v_1,\ldots,v_n)\in\F_p^n\setminus\{(0,\ldots,0)\}$ such that $$f(v_1,\ldots,v_n)\neq 0.$$
Then $\lambda_i(v_1,\ldots,v_n)\neq 0$, and the line $
l(\vec{v},\bc)
$ intersects $\sU_{\fL_i}$ at the unique point
$$
(v_1t_i+c_1,\ldots,v_nt_i+c_n),
$$
where
$$
t_i=-\frac{\mu_i}{\lambda_i(v_1,\ldots,v_n)}.
$$
Clearly, for each $1\leq i<j\leq r$,
\begin{align*}
\mu_j\cdot\lambda_i(v_1,\ldots,v_n)\neq
\mu_i\cdot\lambda_j(v_1,\ldots,v_n)
\end{align*}
implies that $t_i\neq t_j$ and therefore the intersection points are distinct and
we are done.
\end{proof}

\begin{proof}[Proof of Theorem \ref{main}]
Recall that $\tau_p(x)=\langle x\rangle_p$ is the natural homomorphism from $\Z_p$ to $\F_p$.
For each $1\leq i\leq r$, let $\fL_i=\tau_p L_i$ be a linear function over $\F_p^n$.
Since $\{\cU_{L_1},\ldots,\cU_{L_r}\}$ is admissible,
$\sU_{\fL_1},\ldots,\sU_{\fL_r}$ are distinct hyperplanes over $\F_p^n$,
where $\sU_{\fL_i}=\tau_p(\cU_{L_i})$. Let
$$
\sV:=\bigcup_{1\leq i<j\leq r}(\sU_{\fL_i}\cap \sU_{\fL_j}).
$$
Our first claim is the following: if $(s_1,\ldots,s_n)\in\Z_p^n$ is such that
$$\bc:=(\tau_p(s_1),\dots,\tau_p(s_n))\not\in \sV,$$
then $$\Psi(s_1p,\ldots,s_np)\equiv0\pmod{p^r}.$$

\noindent According to Lemma \ref{linehyperplane}, there exists $$
\vec{v}=(v_1,\ldots,v_n)\in\F_p^n\setminus\{(0,\ldots,0)\},
$$ such that, for $1\leq i\leq r$,
$$
l(\vec{v},\bc)\cap\sU_{\fL_i}=\{(v_1t_i+c_1,\ldots,v_nt_i+c_n)\},
$$
where $t_1,\ldots,t_r\in\F_p$ are distinct.
Without loss of generality, we may assume that $v_1,\ldots,v_n\in\{0,1,\ldots,p-1\}$. Then
$$
\cL:=\{(v_1t+s_1,\ldots,v_nt+s_n):\,t\in\Z_p\}
$$
forms a line over $\Z_p^n$, too. Now for each $1\leq i\leq r$, $\cL$ intersects with $\cU_{L_i}$
at the single point
$$
(v_1a_i+s_1,\ldots,v_na_i+s_n),
$$
where $a_i\in\Z_p$ and $\tau_p(a_i)=t_i$.
Let
$$
\Omega(x):=\Psi(v_1x+s_1p,\ldots,v_nx+s_np).
$$
Since $(v_1a_i+s_1,\ldots,v_na_i+s_n)\in\cU_{L_i}$, it follows
$$
\Omega(a_ip)=\Psi(v_1a_ip+s_1p,\ldots,v_na_ip+s_np)\equiv 0\pmod{p^r}
$$
for each $1\leq i\leq r$, and by applying Theorem \ref{psiaiprt}, we find that
$$
\Omega(tp)\equiv 0\pmod{p^r}
$$
for any $t\in\Z_p$. In particular,
$$
\Psi(s_1p,\ldots,s_np)=\Omega(0)\equiv0\pmod{p^r}
$$
and our claim is proved.\medskip

\noindent It remains to show that if $(s_1,\ldots,s_n)\in\Z_p^n$ is such that
$$\bc:=(\tau_p(s_1),\dots,\tau_p(s_n))\in \sV,$$
then $$\Psi(s_1p,\ldots,s_np)\equiv0\pmod{p^r}.$$

\noindent We arbitrarily choose $\vec{v}:=(v_1,\ldots,v_n)\in\F_p^n$, where $v_i\in\{0,1,\ldots,p-1\}$, such that $l(\vec{v},\bc)$ is not parallel to any of $\sU_{\fL_1},\ldots,\sU_{\fL_r}$.
Note that
$$
\sV=\bigcup_{1\leq i<j\leq r}(\sU_{\fL_i}\cap \sU_{\fL_j})\subseteq
\bigcup_{i=1}^r\sU_{\fL_i}.
$$
So we must have
$$
|l(\vec{v},\bc)\cap\sV|\leq r-1.
$$
Since $p\geq 2r-1$, we may find distinct $t_1,\ldots,t_{r}\in\F_p$ so that, for each $1\leq i\leq r$,
$$
(v_1t_i+c_1,\ldots,v_nt_i+c_n)\not\in\sV.
$$
By the previous claim, for $1\leq i\leq r$, if $a_i\in\Z_p$ is such that $\tau_p(a_i)=t_i$, then
$$\Psi((v_1a_i+s_1)p,\ldots,(v_na_i+s_n)p)\equiv 0\pmod{p^r}.$$
Let
$$\Omega(x):=\Psi(v_1x+s_1p,\ldots,v_nx+s_np).$$
Then  $\Omega(a_ip)\equiv0\pmod{p^r}$ for  $1\leq i\leq r$,
and, again by Theorem \ref{psiaiprt},
$$\Omega(tp)\equiv0\pmod{p^r}$$
for any $t\in\Z_p$. In particular,
$$\Psi(s_1p,\ldots,s_np)=\Omega(0)\equiv0\pmod{p^r}.$$
\end{proof}

\section{The Taylor expansions of rational functions and of the $p$-adic Gamma function}
\setcounter{equation}{0}\setcounter{theorem}{0}
\setcounter{lemma}{0}\setcounter{corollary}{0}

Let $\Q_p$ be the quotient field of $\Z_p$, i.e., $\Q_p=\{a/b:\,a,b\in\Z_p,\ b\neq 0\}$.
For each $x\in\Z_p$, let $\nu_p(x)$ denote the $p$-adic order of $x$, i.e.,
$$\nu_p(x):=\max\{n\in\Z:\, xp^{-n}\in\Z_p\}.
$$

Let $f:\,\Z_p\to\Z_p$ and $\alpha\in\Z_p$. Suppose that there exist $A_1(\alpha),\ldots,A_{r}(\alpha)\in\Q_p$ such that
$$
f(\alpha+tp)=f(\alpha)+\sum_{k=1}^{\infty}\frac{A_{k}(\alpha)}{k!}\cdot(tp)^k
$$
for any $t\in\Z_p$, i.e.,
$$
f(\alpha+tp)\equiv f(\alpha)+\sum_{k=1}^{r-1}\frac{A_{k}(\alpha)}{k!}\cdot(tp)^k\pmod{p^{r-w(\alpha,r)}}
$$
for each $r\geq 1$, where $w(\alpha,r)\geq 0$ is an integer only depending on $\alpha,r$, and $r-w(\alpha,r)$ tends to infinity as $r\to\infty$.
Then we say that $f$ has the {\it Taylor series} at $x=\alpha$ with respect to the function $w$.
In view of \eqref{Taylorexpansionorderr}, if $w(\alpha,r+1)=0$ and $A_1(\alpha),\ldots,A_r(\alpha)\in\Z_p$, then $f$ has a strong Taylor expansion of order $r$ at $\alpha$.

The Taylor series of $f$ at $\alpha$ must be unique. In fact, assume on the contrary that the zero function has a Taylor series
$$
0=\sum_{k=k_0}^\infty\frac{A_k}{k!}\cdot(tp)^k,
$$
where $A_{k_0}\neq 0$. Note that $k+\nu_p(A_k)-\nu_p(k!)$ tends to infinity as $k\to\infty$.  We can choose a sufficiently large $r_0$ and set $t=p^{r_0}$ such that
$$
\sum_{k=k_0+1}^\infty\frac{A_k}{k!}\cdot(p^{r_0}\cdot p)^k\equiv 0\pmod{p^{k_0(r_0+1)+\nu_p(A_{k_0})-\nu_p(k_0!)+1}},
$$
i.e.,
$$
0\equiv \frac{A_{k_0}}{k_0!}\cdot(p^{r_0}\cdot p)^{k_0}\pmod{p^{k_0(r_0+1)+\nu_p(A_{k_0})-\nu_p(k_0!)+1}}.
$$
This leads to an evident contradiction.

In this section, we shall study the Taylor series of rational functions and of the $p$-adic Gamma function.
Let $\Z_p[x]$ denote the ring of all polynomials whose coefficients lie in $\Z_p$. The next lemma shows the existence of the Taylor series of rational functions.
\begin{lemma}
Suppose that $P(x),Q(x)\in\Z_p[x]$ and $\alpha\in\Z_p$. If $Q(\alpha)\not\equiv0\pmod{p}$, then the rational function $P(x)/Q(x)$ has the Taylor series at $x=\alpha$.
\end{lemma}
\begin{proof}
It suffices to prove that $1/Q(x)$ has the Taylor expansion of order $r$ at $x=\alpha$. Let $H(x)=(Q(\alpha+x)-Q(\alpha))/x$. Then for each $r\geq 1$
\begin{align*}
\frac1{Q(\alpha+tp)}&=\frac1{Q(\alpha)+tp\cdot H(tp)}\\
&\equiv
\frac1{Q(\alpha)^{r}}\cdot\frac{Q(\alpha)^{r}-(tp\cdot H(tp))^r}{Q(\alpha)+tp\cdot H(tp)}
=\sum_{k=0}^{r-1}\frac{(-1)^kH(tp)^k}{Q(\alpha)^{k+1}}\cdot(tp)^k \pmod{p^{r}}.
\end{align*}
\end{proof}
\begin{Rem}
In fact, if $Q(\alpha)\not\equiv0\pmod{p}$, then it is not hard to verify that $f(x)=P(x)/Q(x)$ has the Taylor expansion at $x=\alpha$,
$$
f(\alpha+tp)\equiv f(\alpha)+\sum_{k=1}^{r-1}\frac{d^kf(\alpha)}{d x^k}\cdot\frac{(tp)^k}{k!}\pmod{p^{r}}.
$$
\end{Rem}

The Taylor series of $p$-adic Gamma function is a little bit more involved. For $\alpha\in\Z_p$, let $G_k(\alpha):=\Gamma_p^{(k)}(\alpha)/\Gamma_p(\alpha)$, where $\Gamma_p^{(k)}(\alpha)$ is the $k$-th derivative of $\Gamma_p$.
\begin{lemma}
For each $\alpha,t\in\Z_p$, there exists the convergent series in the sense of $p$-adic norm
\begin{equation}\label{Gammaptptaylor}
\Gamma_p(\alpha+tp)=\Gamma_p(\alpha)\sum_{k=0}^{\infty}\frac{G_k(\alpha)}{k!}\cdot(tp)^k.
\end{equation}
Furthermore, for each $k\geq 0$,
\begin{equation}\label{nupGkx}
\nu_p\big(G_k(0)\big)\geq-\bigg\lfloor\frac{k}{p}\bigg\rfloor,
\end{equation}
where $\lfloor x\rfloor:=\max\{n\in\Z:\,n\leq x\}$.
\end{lemma}
\begin{proof}
This lemma follows from \cite[Lemma 11]{LoRa16}.
\end{proof}
\begin{corollary}\label{Gammapalphatpprc} For each $r\geq 1$,
\begin{equation}\label{Gammapalphatppr}
\Gamma_p(\alpha+tp)\equiv\Gamma_p(\alpha)\sum_{k=0}^{r-1}\frac{G_k(\alpha)}{k!}\cdot(tp)^k\pmod{p^{r-\sigma_r}},
\end{equation}
where
$$
\sigma_r:=\max_{k\geq r}\bigg(\nu_p(k!)+\bigg\lfloor\frac kp\bigg\rfloor-k+r\bigg).
$$
\end{corollary}
Since $\nu_p(k!)<k/(p-1)$, we have $$\sigma_r<r\cdot\frac{2p-1}{(p-1)p}.$$
Moreover it is easy to check that $\tau_r=0$ for each $1\leq r\leq p-2$.

Below we shall give an explicit formula of $G_k(\alpha)$ for each $\alpha\in\Z_p$.
Let
\begin{equation}\label{Hnspdef}
H_n^{(s)}(p):=\sum_{\substack{1\leq k\leq n\\ p\nmid k}}\frac1{k^s}
\end{equation}
and
\begin{equation}\label{fHnspdef}
\fH_n^{(s)}(p):=\sum_{\substack{1\leq k_1<k_2<\cdots<k_s\leq n\\
p\nmid k_1k_2\cdots k_s
}}\frac1{k_1k_2\cdots k_s}.
\end{equation}
In particular, we set $\fH_n^{(0)}(p)=1$. Clearly $H_n^{(1)}(p)=\fH_n^{(1)}(p)$.
The following lemma gives some useful arithmetic properties of $H_n^{(s)}(p)$.
\begin{lemma}\label{Hnmspl} Let $p$ be prime and $r\geq 1$.
Suppose that $n,m$ are nonnegative integers with $n\equiv m\pmod{p^r}$. Then
for any $s\geq 1$,
\begin{equation}\label{Hnmsp}
H_{n}^{(s)}(p)\equiv H_m^{(s)}\begin{cases}
\pmod{p^r},&\text{if }p-1\nmid s,\\
\pmod{p^{r-\nu_p(s)-1}},&\text{if }p-1\mid s.\\
\end{cases}
\end{equation}
Furthermore, if $s\geq 1$ is odd and
$n+m+1\equiv 0\pmod{p^r}$, then
\begin{equation}\label{Hnmspodd}
H_{n}^{(s)}(p)\equiv H_{m}^{(s)}(p)\pmod{p^r}.
\end{equation}
\end{lemma}
In view of \eqref{Hnmsp}, for each $\alpha\in\Z_p$ and $s\geq 1$, we may define
\begin{equation}
H_\alpha^{(s)}(p):=\lim_{\substack{n\in\N\\ |n-\alpha|_p\to 0}}H_n^{(s)}.
\end{equation}

Similarly, for $\fH_n^{(s)}(p)$, we have
\begin{lemma}
\label{fHnmspl} Let $p\geq 3$ be prime and $r\geq 1$.
Suppose that $n,m$ are nonnegative integers with $n\equiv m\pmod{p^r}$. Then
for any $1\leq s<p^r$,
\begin{equation}\label{fHnmsp}
\fH_{n}^{(s)}(p)\equiv\fH_{m}^{(s)}(p)\pmod{p^{r-\eta_s}},
\end{equation}
where
$$
\eta_s:=\bigg\lfloor\frac{s}{p-1}\bigg\rfloor+\nu_p\bigg(\bigg\lfloor\frac{s}{p-1}\bigg\rfloor!\bigg).
$$
\end{lemma}
So, for each $\alpha\in\Z_p$, define
\begin{equation}
\fH_\alpha^{(s)}(p):=\lim_{\substack{n\in\N\\ |n-\alpha|_p\to 0}}\fH_n^{(s)}.
\end{equation}
\begin{theorem}\label{Gkalphat}
Let $p\geq 3$ be prime and let $G_0(0),G_1(0),\ldots,$ be the functions given in \eqref{Gammaptptaylor}. For each $\alpha\in\Z_p$ and $k\geq 1$,
\begin{equation}\label{Gkalpha}
G_k(\alpha)=k!\sum_{j=0}^{k}\frac{G_j(0)}{j!}\cdot \fH_{\alpha-1}^{(k-j)}(p).
\end{equation}
In particular, if $1\leq r\leq p-2$, then
\begin{equation}\label{Gkalpha2}
G_k(\alpha)\equiv k!\sum_{j=0}^{k}\frac{G_j(0)}{j!}\cdot \fH_{a_{r-j}-1}^{(k-j)}(p)\pmod{p^r}
\end{equation}
for any $1\leq k\leq r-1$, where $a_i=p^{i}-\langle-\alpha\rangle_{p^{i}}$.
\end{theorem}
Although we believe that \eqref{Gkalpha} is not new, we could not find it in the literature. Hence, for the sake of completeness, the proofs of Lemma \ref{Hnmspl},
Lemma \ref{fHnmspl} and Theorem \ref{Gkalphat} will be given in the appendix.

\section{The Dixon type ${}_3F_2$ supercongruence modulo $p^3$}
\setcounter{equation}{0}\setcounter{theorem}{0}
\setcounter{lemma}{0}\setcounter{corollary}{0}

For convenience, for each $\alpha\in\Z_p$, we introduce the notation
\begin{equation}
\alpha_p^*:=\frac{\alpha+\langle-\alpha\rangle_p}{p}.
\end{equation}
Clearly $\alpha_p^*\in\Z_p$.

The Dixon's well-poised summation formula \cite[Theorem 3.4.1]{AAR99} asserts that
\begin{align}\label{dixonwellpoisedsum}
&{}_3F_2\bigg[\begin{matrix}
\alpha&\beta&\gamma\\
&\alpha-\beta+1&\alpha-\gamma+1
\end{matrix}\bigg|\,1\bigg]\notag\\
&\quad=\frac{\Gamma(\frac12\alpha+1)\Gamma(\alpha-\beta+1)\Gamma(\alpha-\gamma+1)\Gamma(\frac12\alpha-\beta-\gamma+1)}{\Gamma(\alpha+1)\Gamma(\frac12\alpha-\beta+1)\Gamma(\frac12\alpha-\gamma+1)\Gamma(\alpha-\beta-\gamma+1)},
\end{align}
where $\Re(\alpha-2\beta-2\gamma)>-2$ or $\alpha\leq 0$ is an integer.
Setting $\alpha=\beta=\gamma$ in \eqref{dixonwellpoisedsum}, we get
\begin{align}\label{dixonalpha}
{}_3F_2\bigg[\begin{matrix}
\alpha&\alpha&\alpha\\
&1&1
\end{matrix}\bigg|\,1\bigg]=\frac{\Gamma(1+\frac12\alpha)\Gamma(1-\frac32\alpha)}{\Gamma(1+\alpha)\Gamma(1-\alpha)\Gamma(1-\frac12\alpha)^2}.
\end{align}
The next theorem is a $p$-adic analogue of \eqref{dixonalpha}.
\begin{theorem}\label{dixonalphatheorem}
Suppose that $p$ is an odd prime and $\alpha\in\Z_p^\times$.
Let
$$g_p(\alpha)=\frac{\Gamma_p(1+\frac12\alpha)\Gamma_p(1-\frac32\alpha)}{\Gamma_p(1+\alpha)\Gamma_p(1-\alpha)\Gamma_p(1-\frac12\alpha)^2}.$$
Then the following congruence holds modulo $p^3$,
$$
{}_3F_2\bigg[\begin{matrix}
\alpha&\alpha&\alpha\\
&1&1
\end{matrix}\bigg|\,1\bigg]_{p-1}
\!\!\!\!\equiv
\begin{cases}
2g_p(\alpha)&\text{if $\langle-\alpha\rangle_p$ is even and $\langle-\alpha\rangle_p<\frac23p$,}
\vspace{2mm}\\
p(2-3\alpha_p^*)\cdot g_p(\alpha)&\text{if $\langle-\alpha\rangle_p$ is even and $\langle-\alpha\rangle_p\geq\frac23p$,}
\vspace{2mm}\\
p\alpha_p^*\cdot g_p(\alpha)&\text{if $\langle-\alpha\rangle_p$ is odd and $\langle-\alpha\rangle_p<\frac13p$,}
\vspace{2mm}\\
\frac{1}{2}p^2\alpha_p^*(1-3\alpha_p^*)\cdot g_p(\alpha)&\text{if $\langle-\alpha\rangle_p$ is odd and $\langle-\alpha\rangle_p\geq\frac13p$.}
\end{cases}
$$
\end{theorem}

Here we only prove the first case where $\langle-\alpha\rangle_p$ is even and $\langle-\alpha\rangle_p<2p/3$, since the proofs of the other cases are very similar. Clearly  it can be verified for $p\leq 5$ and for each $1\leq \alpha\leq p^3$ via numerical computations. So below we assume that $p\geq 7$.

Let $a=\langle-\alpha\rangle_p$. According to our assumptions, $a$ is even and $a<2p/3$.
Let
\begin{align*}
&\Psi(x,y,z):={}_3F_2\bigg[\begin{matrix}-a+x&-a+y&-a+z\\ &1+x-y&1+x-z\end{matrix}\bigg|\,1\bigg]_{p-1}\\
&-\frac{2\Gamma_p(1-\frac12a+\frac12x)\Gamma_p(1+x-y)\Gamma_p(1+x-z)\Gamma_p(1+\frac32a+\frac12x-y-z)}{\Gamma_p(1-a+x)\Gamma_p(1+\frac12a+\frac12x-y)\Gamma_p(1+\frac12a+\frac12x-y)\Gamma_p(1+a+x-y-z)}.
\end{align*}
Clearly the congruence
$${}_3F_2\bigg[\begin{matrix}
\alpha&\alpha&\alpha\\
&1&1
\end{matrix}\bigg|\,1\bigg]_{p-1}\equiv
2g_p(\alpha)\pmod{p^3}$$
 is equivalent to
\begin{equation}
\Psi(p\alpha_p^*,p\alpha_p^*,p\alpha_p^*)\equiv0\pmod{p^3}.
\end{equation}
By Theorem \ref{main}, it suffices to prove the following lemma.

\begin{lemma}\label{dixonPsirst0} Suppose that $r,s,t\in\Z_p$. Then
\begin{equation}
\Psi(rp,sp,tp)=0
\end{equation}
provided that at least one of $r,s,t$ is zero.
\end{lemma}
\begin{proof}
First, we shall prove that
\begin{equation}\label{dixonPsi0st}
\Psi(0,sp,tp)=0
\end{equation}
for each $s,t\in\Z_p$. In fact, we may assume that $sp,tp\in\Q\setminus\Z$, i.e., both $s$ and $t$ are non-integral rational number. For each $m\geq 1$ and $x\in\Z_p$, let
\begin{equation}\label{xmxpm}
x_m=\langle x\rangle_{p^m}-1+\frac{1}{1+p^m}.
\end{equation}
Clearly $x_mp\in\Q\setminus\Z$ and
$
x_m\equiv x\pmod{p^m}$.
Then in the sense of $p$-adic norm,
$$
\lim_{m\to\infty}x_m=x.
$$
So
$$
\Psi(0,sp,tp)=\lim_{m\to\infty}\Psi(0,s_mp,t_mp).
$$
Thus it suffices to show that $\Psi(0,s_mp,t_mp)=0$ for each $m\geq 1$. Of course, $1/(1+p^m)$ in \eqref{xmxpm} can be replaced by $c/(c+dp^m)$ for arbitrary $c,d\in\Z$ with $p\nmid cd$.

Below we assume that $sp,tp\in\Q\setminus\Z$. Furthermore, in view of \eqref{xmxpm}, we may also assume that $sp+tp\in\Q\setminus\Z$. By \eqref{dixonwellpoisedsum},
we have
\begin{align*}
&{}_3F_2\bigg[\begin{matrix}-a&-a+sp&-a+tp\\ &1-sp&1-tp\end{matrix}\bigg|\,1\bigg]_{p-1}=\lim_{z\to 0}{}_3F_2\bigg[\begin{matrix}-a+z&-a+sp&-a+tp\\ &1+z-sp&1+z-tp\end{matrix}\bigg|\,1\bigg]\\
&=\frac{\Gamma(1-sp)}{\Gamma(1+\frac12a-sp)}\cdot\frac{\Gamma(1-tp)}{\Gamma(1+\frac12a-tp)}\cdot\frac{\Gamma(1+\frac32a-sp-tp)}{\Gamma(1+a-sp-tp)}\cdot\lim_{z\to 0}\frac{\Gamma(1-\frac12a+\frac12z)}{\Gamma(1-a+z)}.
\end{align*}
Since $a$ is even and $a<2p/3$,
$$
\frac{\Gamma(1+\frac32a-sp-tp)}{\Gamma(1+a-sp-tp)}=\prod_{j=0}^{\frac12a-1}(1+a-sp-tp+j)
=(-1)^{\frac12a}\cdot\frac{\Gamma_p(1+\frac32a-sp-tp)}{\Gamma_p(1+a-sp-tp)}.
$$
Similarly, we have
$$
\frac{\Gamma(1-sp)}{\Gamma(1+\frac12a-sp)}=\frac{(-1)^{\frac12a}\cdot\Gamma_p(1-sp)}{\Gamma_p(1+\frac12a-sp)},\quad\frac{\Gamma(1-tp)}{\Gamma(1+\frac12a-tp)}=\frac{(-1)^{\frac12a}\cdot\Gamma_p(1-tp)}{\Gamma_p(1+\frac12a-tp)}.
$$
Furthermore, recall that for each nonnegative integer $n$, $\Gamma(z)$ has a simple pole at $z=-n$ with the residue
$$
\lim_{z\to-n}(z+n)\Gamma(z)=\frac{(-1)^n}{n!}.
$$
So
$$
\lim_{z\to 0}\frac{\Gamma(1-\frac12a+\frac12z)}{\Gamma(1-a+z)}=\lim_{z\to 0}\frac{(-1)^{\frac12a-1}(\frac12a-1)!}{(-1)^{a-1}(a-1)!}\cdot\frac{z}{\frac12z}=\frac{2\Gamma_p(\frac12a)}{\Gamma_p(a)}=(-1)^{\frac12a}\cdot\frac{2\Gamma_p(1-\frac12a)}{\Gamma_p(1-a)}.
$$
Thus
\begin{align*}
&{}_3F_2\bigg[\begin{matrix}-a&-a+sp&-a+tp\\ &1-sp&1-tp\end{matrix}\bigg|\,1\bigg]_{p-1}={}_3F_2\bigg[\begin{matrix}-a&-a+sp&-a+tp\\ &1-sp&1-tp\end{matrix}\bigg|\,1\bigg]\\
&\quad =\frac{\Gamma_p(1-sp)}{\Gamma_p(1+\frac12a-sp)}\cdot\frac{\Gamma_p(1-tp)}{\Gamma_p(1+\frac12a-tp)}\cdot\frac{\Gamma_p(1+\frac32a-sp-tp)}{\Gamma_p(1+a-sp-tp)}\cdot\frac{2\Gamma_p(1-\frac12a)}{\Gamma_p(1-a)}.
\end{align*}
The proof of \eqref{dixonPsi0st} is concluded.

Similarly, noting that
$$
\frac{\Gamma(1+b_1+rp)}{\Gamma(-b_2+rp)}=(-1)^{b_1+b_2+1}rp\cdot\frac{\Gamma_p(1+b_1+rp)}{\Gamma_p(-b_2+rp)}
$$
for each $0\leq b_1,b_2\leq p-1$,
we can obtain that
\begin{align*}
&{}_3F_2\bigg[\begin{matrix}-a+rp&-a&-a+tp\\ &1+rp&1+rp-tp\end{matrix}\bigg|\,1\bigg]_{p-1}={}_3F_2\bigg[\begin{matrix}-a+rp&-a&-a+tp\\ &1+rp&1+rp-tp\end{matrix}\bigg|\,1\bigg]\\
&=\frac{\Gamma(1+rp)}{\Gamma(1-a+rp)}\cdot\frac{\Gamma(1-\frac12a+\frac12rp)}{\Gamma(1+\frac12a+\frac12rp)}\cdot\frac{\Gamma(1+rp-tp)}{\Gamma(1+a+rp-tp)}\cdot\frac{\Gamma(1+\frac32a+\frac12rp-tp)}{\Gamma(1+\frac12a+\frac12rp-tp)}\\
&=\frac{rp\cdot\Gamma_p(1+rp)}{\Gamma_p(1-a+rp)}\cdot\frac{\Gamma_p(1-\frac12a+\frac12rp)}{\frac12rp\cdot\Gamma_p(1+\frac12a+\frac12rp)}\cdot\frac{\Gamma_p(1+rp-tp)}{\Gamma_p(1+a+rp-tp)}\cdot\frac{\Gamma_p(1+\frac32a+\frac12rp-tp)}{\Gamma_p(1+\frac12a+\frac12rp-tp)},
\end{align*}
which evidently implies
$$
\Psi(rp,0,tp)=0.
$$
Symmetrically, we also have $\Psi(rp,sp,0)=0$.
\end{proof}

\section{The Watson-Whipple type ${}_3F_2$  supercongruence modulo $p^3$}
\setcounter{equation}{0}\setcounter{theorem}{0}
\setcounter{lemma}{0}\setcounter{corollary}{0}

Let us consider  the Watson identity \cite[Theorem 3.5.5 (i)]{AAR99}
\begin{align}\label{watsonsum}
&{}_3F_2\bigg[\begin{matrix}
\alpha&\beta&\gamma\notag\\
&\frac12(\alpha+\beta+1)&2\gamma
\end{matrix}\bigg|\,1\bigg]\\
&\quad =\frac{\Gamma(\frac12)\Gamma(\frac12+\gamma)\Gamma(\frac12(\alpha+\beta+1))\Gamma(\frac12(1-\alpha-\beta)+\gamma)}{\Gamma(\frac12(1+\alpha))\Gamma(\frac12(1+\beta))\Gamma(\frac12(1-\alpha)+\gamma)\Gamma(\frac12(1-\beta)+\gamma)},
\end{align}
where $\Re(2\gamma-\alpha-\beta)>-1$ or $\alpha\leq 0$ is an integer. Setting
$\beta=\alpha$ and $\gamma=1/2$ in \eqref{watsonsum}, we obtain that
\begin{align}\label{watsonalpha}
{}_3F_2\bigg[\begin{matrix}
\alpha&1-\alpha&\frac12\\
&1&1
\end{matrix}\bigg|\,1\bigg]=\frac{\Gamma(\frac12)^2}{\Gamma(1-\frac12\alpha)^2\Gamma(\frac12+\frac12\alpha)^2}.
\end{align}
Then we also have a $p$-adic analogue of \eqref{watsonalpha} as follows.
\begin{theorem}\label{watsonalphatheorem}
Let $p$ be an odd prime and let $\alpha\in\Z_p$. If $\langle-\alpha\rangle_p$ is even, then
\begin{equation}\label{watson3F2e1}
{}_3F_2\bigg[\begin{matrix}\alpha&1-\alpha&\frac{1}{2}\\&1&1\end{matrix}\bigg|\ 1\bigg]_{p-1}\equiv \frac{\Gamma_p(\frac{1}{2})^2}{\Gamma_p(1-\frac{1}{2}\alpha)^2\Gamma_p(\frac{1}{2}+\frac{1}{2}\alpha)^2}\pmod{p^3}.
\end{equation}
On the other hand, if $\langle-\alpha\rangle_p$ is odd, then
\begin{equation}\label{watson3F2e2}
{}_3F_2\bigg[\begin{matrix}\alpha&1-\alpha&\frac{1}{2}\\&1&1\end{matrix}\bigg|\ 1\bigg]_{p-1}\equiv\frac{p^2\alpha_p^*(\alpha_p^*-1)}{4}\cdot\frac{\Gamma_p(\frac{1}{2})^2}{\Gamma_p(1-\frac{1}{2}\alpha)^2\Gamma_p(\frac{1}{2}+\frac{1}{2}\alpha)^2}\pmod{p^3}.
\end{equation}
\end{theorem}
We mention that \eqref{watson3F2e1} was conjectured by Liu in \cite{Liu18}.

Here we only give the proof of \eqref{watson3F2e1}, too.
Verify \eqref{watson3F2e1} for $p=3,5$ directly and assume that $p\geq 7$.
Let $a$ denote $\langle-\alpha\rangle_p$. Let
\begin{align*}
\Omega(x,y,z):=&{}_3F_2\bigg[\begin{matrix}-a+x&1+a-x&\frac{1}{2}(1-y)\\&1+z&1-y-z\end{matrix}\bigg|\ 1\bigg]_{p-1}\\
&\quad-\frac{\Gamma_p(\frac{1+z}{2})\Gamma_p(1+\frac z2)\Gamma_p(\frac{1-y-z}{2})\Gamma_p(1-\frac{y+z}{2})}{\Gamma_p(\frac{1-a}{2}+\frac{x+z}{2})\Gamma_p(\frac{1-a}{2}+\frac{x-y-z}{2})\Gamma_p(1+\frac{a}{2}+\frac{z-x}{2})\Gamma_p(1+\frac{a}{2}-\frac{x+y+z}{2})}.
\end{align*}
Then \eqref{watson3F2e1} is equivalent to
$$
\Omega(p\alpha_p^*,0,0)\equiv0\pmod{p^3}.
$$
In view of Theorem 2.1, it suffices to show the next lemma.

\begin{lemma}
$$
\Omega(rp,p,tp)=\Omega(0,sp,tp)=\Omega(p,sp,tp)=0
$$
for each $r,s,t\in\Z_p$.
\end{lemma}
\begin{proof}
First, we shall prove
\begin{equation}\label{watsonomegarpptp}
\Omega(rp,p,tp)=0.
\end{equation} Note that for each $m\geq 0$, we may choose $1\leq r_m\leq 2p^m$ with $2\mid r$ such that $r_m\equiv r\pmod{p^m}$, i.e.,
$$
\lim_{m\to\infty}\Omega(r_mp,p,tp)=\Omega(rp,p,tp).
$$
So without loss of generality, we may assume that $r$ is a positive even integer.
Furthermore, assume that $tp\in\Q\setminus\Z$.

However, the Watson identity is not suitable to prove \eqref{watsonomegarpptp}. We need another formula due to Whipple (cf. \cite[Theorem 3.5.5 (ii)]{AAR99}):
\begin{align}\label{WhippleSum}
&{}_3F_2\bigg[\begin{matrix}\alpha&1-\alpha&\beta\\&\gamma&2\beta-\gamma+1\end{matrix}\bigg|\ 1\bigg]\notag\\
&\quad=\frac{2^{1-2\beta}\pi\Gamma(\gamma)\Gamma(2\beta-\gamma+1)}{\Gamma(\beta+\frac{1}{2}(\alpha-\gamma+1))\Gamma(\beta+1-\frac{1}{2}(\alpha+\gamma))\Gamma(\frac{1}{2}(\alpha+\gamma))\Gamma(\frac{1}{2}(\gamma-\alpha+1))},
\end{align}
where $\Re(\beta)>0$. Clearly Whipple's identity \eqref{WhippleSum} also implies \eqref{watsonalpha} by setting $\beta=1/2$ and $\gamma=1$.

Now according to the Gauss multiplication formula (cf. \cite[p. 371]{Robert00})
$$
\Gamma(z)\Gamma\bigg(z+\frac 12\bigg)=2^{1-2z}\sqrt{\pi}\Gamma(2z),
$$
\eqref{WhippleSum} can be rewritten as
\begin{align}\label{WhippleSum2}
&{}_3F_2\bigg[\begin{matrix}\alpha&1-\alpha&\beta\\&\gamma&2\beta-\gamma+1\end{matrix}\bigg|\ 1\bigg]\notag\\
&\quad=\frac{\Gamma(\gamma)\Gamma(\frac12+\gamma)\Gamma(\frac12+\beta-\frac12\gamma)\Gamma(1+\beta-\frac12\gamma)}{\Gamma(\beta+\frac{1}{2}(\alpha-\gamma+1))\Gamma(\beta+1-\frac{1}{2}(\alpha+\gamma))\Gamma(\frac{1}{2}(\alpha+\gamma))\Gamma(\frac{1}{2}(\gamma-\alpha+1))}.
\end{align}
Applying \eqref{WhippleSum2}, we have
\begin{align*}
&{}_3F_2\bigg[\begin{matrix}-a+rp&1+a-rp&\frac{1}{2}(1-p)\\&1+tp&1-p-tp\end{matrix}\bigg|\ 1\bigg]\\
&=\frac{\Gamma(\frac{1}{2}+\frac{tp}{2})\Gamma(1+\frac{tp}{2})\Gamma(\frac{1}{2}-\frac{p}{2}-\frac{tp}{2})\Gamma(1-\frac{p}{2}-\frac{tp}{2})}{\Gamma(\frac{1}{2}-\frac a2+\frac{rp}{2}-\frac p2-\frac{tp}{2})\Gamma(1+\frac{a}{2}-\frac p2-\frac{rp}{2}-\frac{tp}{2})\Gamma(\frac{1}{2}-\frac a2+\frac{rp}{2}+\frac{tp}{2})\Gamma(1+\frac{a}{2}-\frac{rp}{2}+\frac{tp}{2})}.
\end{align*}
Since both $a$ and $r$ are positive even integers and $a<p$,
\begin{align*}
\frac{\Gamma(1+\frac{a}{2}+\frac{tp}{2})}{\Gamma(1+\frac{a}{2}-\frac{rp}{2}+\frac{tp}{2})}&=\prod_{j=0}^{\frac{1}{2}rp-1}\bigg(1+\frac{a}{2}+\frac{tp}{2}+j\bigg)\\
&=(-1)^{\frac12rp}\cdot
\frac{\Gamma_p(1+\frac{a}{2}+\frac{tp}{2})}{\Gamma_p(1+\frac{a}{2}-\frac{rp}{2}+\frac{tp}{2})}\cdot \prod_{j=1}^{\frac12r}\bigg(jp-\frac {rp}2+\frac{tp}{2}\bigg).
\end{align*}
Therefore
\begin{align*}
&\frac{\Gamma(1+\frac{tp}{2})}{\Gamma(1+\frac{a}{2}-\frac{rp}{2}+tp)}=\frac{\Gamma(1+\frac{tp}{2})}{\Gamma(1+\frac{a}{2}+tp)}\cdot\frac{\Gamma(1+\frac{a}{2}+\frac{tp}{2})}{\Gamma(1+\frac{a}{2}-\frac{rp}{2}+\frac{tp}{2})}\\
&\quad=(-1)^{\frac12a}\cdot
\frac{\Gamma_p(1+\frac{tp}{2})}{\Gamma_p(1+\frac{a}{2}+\frac{tp}{2})}\cdot
(-1)^{\frac12r}p^{\frac12r}\cdot
\frac{\Gamma_p(1+\frac{a}{2}+\frac{tp}{2})}{\Gamma_p(1+\frac{a}{2}-\frac{rp}{2}+\frac{tp}{2})}\cdot \prod_{j=1}^{\frac12r}\bigg(j-\frac {r}2+\frac{t}{2}\bigg)\\
&\quad=(-1)^{\frac12(a+r)}p^{\frac12r}\cdot\frac{\Gamma_p(1+\frac{tp}{2})}{\Gamma_p(1+\frac{a}{2}-\frac{rp}{2}+\frac{tp}{2})}\cdot\prod_{j=1}^{\frac12r}\bigg(j-\frac {r}2+\frac{t}{2}\bigg).
\end{align*}
Similarly, we have
$$
\frac{\Gamma(1-\frac{p}{2}-\frac{tp}{2})}{\Gamma(1+\frac{a}{2}-\frac{p}{2}-\frac{rp}{2}-\frac{tp}{2})}=(-1)^{\frac12(a+r)}p^{\frac12r}\cdot
\frac{\Gamma_p(1-\frac{p}{2}-\frac{tp}{2})}{\Gamma_p(1+\frac{a}{2}-\frac{p}{2}-\frac{rp}{2}-\frac{tp}{2})}\prod_{j=1}^{\frac12r}\bigg(j-\frac{1}{2}-\frac {r}2-\frac{t}{2}\bigg),
$$
$$
\frac{\Gamma(\frac12-\frac{p}{2}-\frac{tp}{2})}{\Gamma(\frac12+\frac{a}{2}-\frac{p}{2}+\frac{rp}{2}-\frac{tp}{2})}=\frac{(-1)^{\frac12(a+r)}}{p^{\frac12r}}\cdot
\frac{\Gamma_p(\frac12-\frac{p}{2}-\frac{tp}{2})}{\Gamma_p(\frac12+\frac{a}{2}-\frac{p}{2}+\frac{rp}{2}-\frac{tp}{2})}\prod_{j=1}^{\frac12r}\frac1{\frac {r}2-\frac{t}{2}-j},
$$
$$
\frac{\Gamma(\frac12+\frac{tp}{2})}{\Gamma(\frac12+\frac{a}{2}+\frac{rp}{2}+\frac{tp}{2})}=\frac{(-1)^{\frac12(a+r)}}{p^{\frac12r}}\cdot\frac{\Gamma_p(\frac12+\frac{tp}{2})}{\Gamma_p(\frac12+\frac{a}{2}+\frac{rp}{2}+\frac{tp}{2})}\prod_{j=1}^{\frac12r}\frac1{\frac{r}{2}+\frac{t}2+\frac12-j}.
$$
Hence
\begin{align*}
&{}_3F_2\bigg[\begin{matrix}-a+rp&1+a-rp&\frac{1}{2}(1-p)\\&1+tp&1-p-tp\end{matrix}\bigg|\ 1\bigg]_{p-1}\\
&=\frac{\Gamma_p(\frac{1}{2}+\frac{tp}{2})\Gamma_p(1+\frac{tp}{2})\Gamma_p(\frac{1}{2}-\frac{p}{2}-\frac{tp}{2})\Gamma_p(1-\frac{p}{2}-\frac{tp}{2})}{\Gamma_p(\frac{1}{2}-\frac a2+\frac{rp}{2}-\frac p2-\frac{tp}{2})\Gamma_p(1+\frac{a}{2}-\frac p2-\frac{rp}{2}-\frac{tp}{2})\Gamma_p(\frac{1}{2}-\frac a2+\frac{rp}{2}+\frac{tp}{2})\Gamma_p(1+\frac{a}{2}-\frac{rp}{2}+\frac{tp}{2})},
\end{align*}
i.e., \eqref{watsonomegarpptp} is valid.

Next, assume that $sp,tp,(s+t)p\in\Q\setminus\Z$. With help of \eqref{WhippleSum2}, it is not difficult to check that
\begin{align*}
&{}_3F_2\bigg[\begin{matrix}-a&1+a&\frac{1}{2}(1-sp)\\&1+tp&1-sp-tp\end{matrix}\bigg|\ 1\bigg]_{p-1}\notag\\
&\quad=\frac{\Gamma(\frac{1}{2}+\frac{tp}{2})}{\Gamma(\frac{1}{2}-\frac a2+\frac{tp}{2})}\cdot\frac{\Gamma(1+\frac{tp}2)}{\Gamma(1+\frac{a}{2}+\frac{tp}{2})}\cdot
\frac{\Gamma(\frac{1}{2}-\frac{sp}{2}-\frac{tp}{2})}{\Gamma(\frac{1}{2}-\frac a2-\frac{sp}{2}-\frac{tp}2)}\cdot
\frac{\Gamma(1-\frac{sp}2-\frac{tp}{2})}{\Gamma(1+\frac{a}{2}-\frac{sp}{2}-\frac{tp}{2})}\\
&\quad=\frac{\Gamma_p(\frac{1}{2}+\frac{tp}{2})}{\Gamma_p(\frac{1}{2}-\frac a2+\frac{tp}{2})}\cdot\frac{\Gamma_p(1+\frac{tp}2)}{\Gamma_p(1+\frac{a}{2}+\frac{tp}{2})}\cdot
\frac{\Gamma_p(\frac{1}{2}-\frac{sp}{2}-\frac{tp}{2})}{\Gamma_p(\frac{1}{2}-\frac a2-\frac{sp}{2}-\frac{tp}2)}\cdot
\frac{\Gamma_p(1-\frac{sp}2-\frac{tp}{2})}{\Gamma_p(1+\frac{a}{2}-\frac{sp}{2}-\frac{tp}{2})},
\end{align*}
i.e., $\Omega(0,sp,tp)=0$. Similarly, we may get $\Omega(p,sp,tp)=0$.
\end{proof}

\section{The Pfaff-Saalsch\"utz type ${}_3F_2$  supercongruence modulo $p^3$}
\setcounter{equation}{0}\setcounter{theorem}{0}
\setcounter{lemma}{0}\setcounter{corollary}{0}

If $n=\alpha+\beta+1-\gamma-\delta$ is a nonnegative integer, we have
the Pfaff-Saalsch\"utz balanced sum formula (cf. \cite[(1.7.1)]{GR})
\begin{align}\label{PSsum}
{}_3F_2\bigg[\begin{matrix}
\alpha&\beta&-n\\
&\gamma&\delta
\end{matrix}\bigg|\,1\bigg]=\frac{(\gamma-\alpha)_n(\gamma-\beta)_n}{(\gamma)_n(\gamma-\alpha-\beta)_n}.
\end{align}
Setting $\gamma=\delta=1$ in \eqref{PSsum}, we get
\begin{equation}\label{PSalphabeta}
{}_3F_2\bigg[\begin{matrix}
\alpha&\beta&1-\alpha-\beta\\
&1&1
\end{matrix}\bigg|\,1\bigg]=\frac{(1-\alpha)_n(1-\beta)_n}{(1)_n(1-\alpha-\beta)_n},
\end{equation}
where $n=\alpha+\beta-1$ is a nonnegative integer. As a $p$-adic analogue of \eqref{PSalphabeta}, we have the next result.

\begin{theorem}\label{PSalphabetatheorem}
Let $p$ be an odd prime and let $\alpha,\beta\in\Z_p$.
If $\langle-\alpha\rangle_p+\langle-\beta\rangle_p\leq p-1$, then
\begin{equation}\label{PS3F2e1}
{}_3F_2\bigg[\begin{matrix}
\alpha&\beta&1-\alpha-\beta\\&1&1
\end{matrix}\bigg|\ 1\bigg]_{p-1}\equiv\frac{\Gamma_p(1-\alpha-\beta)^2}{\Gamma_p(1-\alpha)^2\Gamma_p(1-\beta)^2}\pmod{p^3}.
\end{equation}
If $\langle-\alpha\rangle_p+\langle-\beta\rangle_p\geq p$, then
\begin{align}\label{PS3F2e2}
{}_3F_2\bigg[\begin{matrix}
\alpha&\beta&1-\alpha-\beta\\&1&1
\end{matrix}\bigg|\ 1\bigg]_{p-1}
&\equiv p^2((\alpha_p^*-1)^2+(\beta_p^*-1)^2+\alpha_p^*\beta_p^*-1)\notag\\
&\qquad\cdot\frac{\Gamma_p(1-\alpha-\beta)^2}{\Gamma_p(1-\alpha)^2\Gamma_p(1-\beta)^2}\pmod{p^3}.
\end{align}
\end{theorem}

Assume that $p\geq 7$.
Let $a=\langle-\alpha\rangle_p$ and $b=\langle-\beta\rangle$.
Assume that $a+b\leq p-1$.
Define
\begin{align*}
\Psi(x,y):=&{}_3F_2\bigg[\begin{matrix}
-a+x&-b+y&1+a+b-x-y\\&1&1
\end{matrix}\bigg|\ 1\bigg]_{p-1}\\
&\quad -\frac{\Gamma_p(1+a+b-x-y)^2}{\Gamma_p(1+a-x)^2\Gamma_p(1+b-y)^2}.
\end{align*}
It is evident that  \eqref{PS3F2e1} is equivalent to
\begin{equation}
\Psi(p\alpha_p^*,p\beta_p^*)\equiv0\pmod{p^3}.
\end{equation}
In view of Theorem \ref{main}, we only need to show the following lemma.

\begin{lemma}
$$
\Psi(sp,tp)=0
$$
provided that $s=0$, or $t=0$, or $s+t=1$.
\end{lemma}
\begin{proof}
Assume that $sp\in\Q\setminus\Z$.
According to \eqref{PSalphabeta}, since $1+a+b\leq p$,
\begin{align*}
&{}_3F_2\bigg[\begin{matrix}
-a+sp&-b+(1-s)p&1+a+b-p\\&1&1
\end{matrix}\bigg|\ 1\bigg]\\
&\qquad=\frac{(1+a-sp)_{p-a-b-1}(1+b-(1-s)p)_{p-a-b-1}}{(1)_{p-a-b-1}(1+a+b-p)_{p-a-b-1}}.
\end{align*}
It is easy to check that
$$
(1+a-sp)_{p-a-b-1}=(-1)^{p-a-b-1}\cdot\frac{\Gamma_p(-b+(1-s)p)}{\Gamma_p(1+a-sp)},$$
$$
(1+b-(1-s)p)_{p-a-b-1}=(-1)^{p-a-b-1}\cdot\frac{\Gamma_p(-a+sp)}{\Gamma_p(1+b-(1-s)p)},
$$
$$
(1)_{p-a-b-1}=(-1)^{p-a-b-1}\cdot\frac{\Gamma_p(p-a-b)}{\Gamma_p(1)},
$$
$$
(1+a+b-p)_{p-a-b-1}=(-1)^{p-a-b-1}\cdot\frac{\Gamma_p(0)}{\Gamma_p(1+a+b-p)}.
$$
It follows that
\begin{align*}
&{}_3F_2\bigg[\begin{matrix}
-a+sp&-b+(1-s)p&1+a+b-p\\&1&1
\end{matrix}\bigg|\ 1\bigg]_{p-1}\\
&\qquad=-\frac{\Gamma_p(-b+(1-s)p)\Gamma_p(-a+sp)\Gamma_p(1+a+b-p)}{\Gamma_p(1+a-sp)\Gamma_p(1+b-(1-s)p)\Gamma_p(p-a-b)}\\
&\qquad=\frac{\Gamma_p(1+a+b-p)^2}{\Gamma_p(1+a-sp)^2\Gamma_p(1+b-(1-s)p)^2},
\end{align*}
where the formula \eqref{Gammapx1x} is used in the second equality. Thus $\Psi(sp,tp)=0$ for each $(s,t)\in\cU_{x+y-1}$.

Similarly, assuming that $tp\in\Q\setminus\Z$, we obtain that
\begin{align*}
{}_3F_2\bigg[\begin{matrix}
-a&-b+tp&1+a+b-tp\\&1&1
\end{matrix}\bigg|\ 1\bigg]&=\frac{(1+b-tp)_{a}(-a-b+tp)_{a}}{(1)_{a}(-a)_{a}}\\
&=-\frac{\Gamma_p(1+a+b-tp)\Gamma_p(-b+tp)\Gamma_p(-a)}{\Gamma_p(1+b-tp)\Gamma_p(-a-b+tp)\Gamma_p(1+a)}\\
&=\frac{\Gamma_p(1+a+b-tp)^2}{\Gamma_p(1+b-tp)^2\Gamma_p(1+a)^2},
\end{align*}
i.e., $\Psi(0,tp)=0$. Symmetrically, we also have $\Psi(sp,0)=0$.
\end{proof}

Assume that $a+b\geq p$. Let
\begin{align*}
\Omega(x,y):=&{}_3F_2\bigg[\begin{matrix}
-a+x&-b+y&1+a+b-x-y\\&1&1
\end{matrix}\bigg|\ 1\bigg]_{p-1}\\
&\quad-((x-1)^2+(y-1)^2+xy-1)\cdot\frac{\Gamma_p(1+a+b-x-y)^2}{\Gamma_p(1+a-x)^2\Gamma_p(1+b-y)^2}.
\end{align*}
Then \eqref{PS3F2e2} follows from the next lemma, whose proof is left to the reader.

\begin{lemma}
$$
\Omega(sp,tp)=0
$$
provided that $s=0$, or $t=0$, or $s+t=1$.
\end{lemma}

\section{The Dougall type supercongruence for ${}_7F_6$ truncated hypergeometric series}
\setcounter{equation}{0}\setcounter{theorem}{0}
\setcounter{lemma}{0}\setcounter{corollary}{0}

A formula of Dougall concerning ${}_7F_6$ hypergeometric series (cf. \cite[Theorem 3.5.1]{AAR99}) says that
\begin{align}\label{Dougall7F6}
&{}_7F_6\bigg[\begin{matrix}\alpha&1+\frac12\alpha&\beta&\gamma&\delta&\epsilon&-n\\
&\frac12\alpha&\alpha-\beta+1&\alpha-\gamma+1&\alpha-\delta+1&\alpha-\epsilon+1&\alpha+n+1\end{matrix}\bigg|\,1\bigg]\notag\\
&\qquad=\frac{(\alpha+1)_n(\alpha-\beta-\gamma+1)_n(\alpha-\beta-\delta+1)_n(\alpha-\gamma-\delta+1)_n}{(\alpha-\beta+1)_n(\alpha-\gamma+1)_n(\alpha-\delta+1)_n(\alpha-\beta-\gamma-\delta+1)_n},
\end{align}
where $n=\beta+\gamma+\delta+\epsilon-2\alpha-1$ is a nonnegative integer. Setting $\epsilon=\alpha$ in \eqref{Dougall7F6}, we have
\begin{align}\label{Dougall7F6b}
&{}_7F_6\bigg[\begin{matrix}\alpha&\alpha&1+\frac12\alpha&\beta&\gamma&\delta&-n\\
&1&\frac12\alpha&\alpha-\beta+1&\alpha-\gamma+1&\alpha-\delta+1&\alpha+n+1\end{matrix}\bigg|\,1\bigg]\notag\\
&\qquad=\frac{(\alpha+1)_n(\alpha-\beta-\gamma+1)_n(\alpha-\beta-\delta+1)_n(\alpha-\gamma-\delta+1)_n}{(\alpha-\beta+1)_n(\alpha-\gamma+1)_n(\alpha-\delta+1)_n(\alpha-\beta-\gamma-\delta+1)_n},
\end{align}
where $n=\beta+\gamma+\delta-\alpha-1$. In \cite{MP17}, Mao and Pan obtained several $\mod p^2$ congruences concerning the $p$-adic analogues of \eqref{Dougall7F6b}. For example, if $\alpha,\beta,\gamma,\delta\in\Z_p$ satisfy $\alpha/\beta=\langle-\alpha\rangle_p/\langle-\beta\rangle_p$ and some additional assumptions, then
\begin{align}\label{MP7F6}
&{}_7F_6\bigg[\begin{matrix}\alpha&\alpha&1+\frac12\alpha&\beta&\gamma&\delta&\epsilon\\
&1&\frac12\alpha&\alpha-\beta+1&\alpha-\gamma+1&\alpha-\delta+1&\alpha-\epsilon+1\end{matrix}\bigg|\,1\bigg]_{p-1}\notag\\
&\quad =-\frac{\alpha}{\alpha-\beta}\cdot\frac{\Gamma_p(1-\beta-\gamma)\Gamma_p(1-\beta-\delta)\Gamma_p(1-\gamma-\delta)}{\Gamma_p(1-\beta)\Gamma_p(1-\gamma)\Gamma_p(1-\delta)\Gamma_p(1-\beta-\gamma-\delta)}\notag\\
&\qquad\cdot\frac{\Gamma_p(\alpha-\beta+1)\Gamma_p(\alpha-\gamma+1)\Gamma_p(\alpha-\delta+1)\Gamma_p(\alpha-\beta-\gamma-\delta+1)}{\Gamma_p(\alpha+1)\Gamma_p(\alpha-\beta-\gamma+1)\Gamma_p(\alpha-\beta-\delta+1)\Gamma_p(\alpha-\gamma-\delta+1)}\pmod{p^2},
\end{align}
where $\epsilon=1+\alpha-\beta-\gamma-\delta$.

In a special case, we have the following $\mod p^5$ extension of \eqref{MP7F6}.
\begin{theorem}\label{Dougall7F6theorem}
Suppose that $d\geq 5$, $d/3<r<d/2$ and $(r,d)=1$. Let $\alpha=r/d$. Then for each prime $p\equiv 1\pmod{d}$,
\begin{align}\label{Dougall7F6alphae1}
{}_7F_6\bigg[\begin{matrix}\alpha&\alpha&\alpha&\alpha&\alpha&1+\frac12\alpha&1-2\alpha\\
&1&1&1&1&\frac12\alpha&3\alpha\end{matrix}\bigg|\,1\bigg]_{p-1}
\equiv\frac{(-1)^{\frac{p-1}{d}}}{3\alpha-1}\cdot\frac{\Gamma_p(\alpha)^5\Gamma_p(3\alpha)}{\Gamma_p(2\alpha)^4}\pmod{p^5}.
\end{align}
\end{theorem}
For example, setting $\alpha=3/8$ in \eqref{Dougall7F6alphae1}, we have
\begin{equation}
{}_7F_6\bigg[\begin{matrix}\frac38&\frac38&\frac38&\frac38&\frac38&\frac{19}{16}&\frac14\\
&1&1&1&1&\frac3{16}&\frac98\end{matrix}\bigg|\,1\bigg]_{p-1}
\equiv 8\cdot (-1)^{\frac{p-1}{8}}\cdot\frac{\Gamma_p(\frac{3}{8})^5\Gamma_p(\frac{9}{8})}{\Gamma_p(\frac{3}{4})^4}\pmod{p^5}
\end{equation}
for each prime $p\equiv 1\pmod{8}$.

Also, by setting $\alpha=2/5$ in \eqref{Dougall7F6alphae1}, for each prime $p\equiv 1\pmod{5}$, we  get
\begin{equation}\label{DFLST16}
{}_5F_4\bigg[\begin{matrix}\frac25&\frac25&\frac25&\frac25&\frac25\\
&1&1&1&1\end{matrix}\bigg|\,1\bigg]_{p-1}
\equiv-\Gamma_p\bigg(\frac{1}{5}\bigg)^5\Gamma_p\bigg(\frac{2}{5}\bigg)^5\pmod{p^5},
\end{equation}
since $5\Gamma_p(\frac65)=-\Gamma_p(\frac15)$ by \eqref{Gammapxx1} and $\Gamma_p(\frac{1}{5})\Gamma_p(\frac{4}{5})=1$ by \eqref{Gammapx1x}.

We note that \eqref{DFLST16} was conjectured by Deines, Fuselier, Long, Swisher and Tu in \cite[(7.4)]{DFLST16}.
In fact, Theorem \ref{Dougall7F6theorem} is a consequence of the following stronger result by letting $\beta=1-2\alpha$.

\begin{theorem}\label{alphabetagammadeltan7F6A}
Let $p\geq 3$ be prime and $\alpha,\beta\in\Z_p$. Suppose that

\medskip\noindent
{\rm (i)} $\langle-\beta\rangle_p<\langle-\alpha\rangle_p$;
\quad
{\rm (ii)} $2\langle-\alpha\rangle_p+\langle-\beta\rangle_p\leq p-1\leq 3\langle-\alpha\rangle_p+2\langle-\beta\rangle_p$;

\medskip\noindent
{\rm (iii)} $\langle-\alpha\rangle_p/\alpha=\langle-\beta\rangle_p/\beta$;
\quad{\rm (iv)} $(\alpha-\beta+1)_{p-1}$ are not divisible by $p^2$.

\medskip\noindent

Then
\begin{align}\label{alphabetagammadeltan7F6Ae}
&{}_7F_6\bigg[\begin{matrix}\alpha&\alpha&\alpha&\alpha&1+\frac12\alpha&\beta&1-\alpha-\beta\\
&1&1&1&\frac12\alpha&\alpha-\beta+1&2\alpha+\beta\end{matrix}\bigg|\,1\bigg]_{M}\notag\\
&\qquad\equiv-\frac{\alpha}{\alpha-\beta}\cdot\frac{\Gamma_p(1-\alpha-\beta)^3\Gamma_p(1-2\alpha)\Gamma_p(1+\alpha-\beta)}{\Gamma_p(1-\beta)^3\Gamma_p(1-\alpha)^3\Gamma_p(1+\alpha)\Gamma_p(1-2\alpha-\beta)}\pmod{p^5},
\end{align}
where
$$
M=2\langle-\alpha\rangle_p+\langle-\beta\rangle_p.
$$
\end{theorem}
Let us explain why Theorem \ref{Dougall7F6theorem}  is a consequence of Theorem \ref{alphabetagammadeltan7F6A}. By substituting $\beta=1-2\alpha$ in Theorem \ref{alphabetagammadeltan7F6A}, we get $M=p-1$ and it follows from \eqref{alphabetagammadeltan7F6Ae} that
\begin{align*}
&{}_7F_6\bigg[\begin{matrix}\alpha&\alpha&\alpha&\alpha&1+\frac12\alpha&1-2\alpha&\alpha\\
&1&1&1&\frac12\alpha&3\alpha&1\end{matrix}\bigg|\,1\bigg]_{p-1}\notag\\
&\qquad\equiv-\frac{\alpha}{3\alpha-1}\cdot\frac{\Gamma_p(\alpha)^3\Gamma_p(1-2\alpha)\Gamma_p(3\alpha)}{\Gamma_p(2\alpha)^3\Gamma_p(1-\alpha)^3\Gamma_p(1+\alpha)\Gamma_p(0)}\\
&\qquad\equiv\frac{(-1)^{\langle -\alpha\rangle_p}}{3\alpha-1}\cdot\frac{\Gamma_p(\alpha)^5\Gamma_p(3\alpha)}{\Gamma_p(2\alpha)^4}\pmod{p^5},
\end{align*}
where it is easy to check that
$
(-1)^{\langle -\alpha\rangle_p}=(-1)^{\frac{(p-1)r}{d}}=(-1)^{\frac{p-1}{d}}.
$

Let
$$
\Psi(x,y,z,w):=
{}_7F_6\bigg[\begin{matrix}\alpha&1+\frac12\alpha&\beta&\gamma&\delta&\epsilon&\rho\\
&\frac12\alpha&\alpha-\beta+1&\alpha-\gamma+1&\alpha-\delta+1&\alpha-\epsilon+1&\alpha-\rho+1\end{matrix}\bigg|\,1\bigg]_M,
$$
where $$
\alpha=-a+ax,\quad\beta=-b+bx,\quad\gamma=-a+y,\quad\delta=-a+z,\quad\epsilon=-a+w,
$$
and
$$
\rho=\beta+\gamma+\delta+\epsilon-2\alpha-1=-1-a-b-2ax+bx+y+z+w.
$$
Furthermore, let
$$
\Psi_*(x,y,z,w):=
{}_7F_6\bigg[\begin{matrix}\alpha&1+\frac12\alpha&\beta&\gamma&\delta&\epsilon&\rho\\
&\frac12\alpha&\alpha-\beta+1&\alpha-\gamma+1&\alpha-\delta+1&\alpha-\epsilon+1&\alpha-\rho+1\end{matrix}\bigg|\,1\bigg].
$$

\begin{lemma}\label{PsixyzwPQ}
$$
\Psi(x,y,z,w)=\frac{P(x,y,z,w)}{Q(x)},
$$
where both $P(x,y,z,w)$ and $Q(x)$ are polynomials over $\Z_p$ and $p\nmid Q(0)$.
\end{lemma}
\begin{proof}
This lemma can be proved in the same way as \cite[Lemma 13.1]{MP17}.
\end{proof}
Lemma \ref{PsixyzwPQ} shows that $\Psi(x,y,z,w)$ is differentiable at each $(x,y,z,w)\in(p\Z_p)^4$.

Let
\begin{align*}
\Omega(x,y,z,w)&=\frac{-a+ax}{-a+b+ax-bx}\cdot\frac{\Gamma_p(-2a-b-ax+bx+y+z+w)}{\Gamma_p(1-a+ax)}\\
&\cdot\frac{\Gamma_p(-a-ax+z+w)}{\Gamma_p(1+b-bx+ax-y)}\cdot\frac{\Gamma_p(-a-ax+y+w)}{\Gamma_p(1+b-bx+ax-z)}\cdot\frac{\Gamma_p(-b-ax+bx+w)}{\Gamma_p(1+a+ax-y-z)}\\
&\cdot\frac{\Gamma_p(1-a+b+ax-bx)}{\Gamma_p(-2a-ax+y+z+w)}\cdot\frac{\Gamma_p(1+ax-y)}{\Gamma_p(-a-b-ax+bx+z+w)}\\
&\cdot\frac{\Gamma_p(1+ax-z)}{\Gamma_p(-a-b-ax+bx+y+w)}\cdot\frac{\Gamma_p(1+a+b+ax-bx-y-z)}{\Gamma_p(-ax+w)}.
\end{align*}
By Theorem \ref{main}, we only need to show the next result.

\begin{lemma}
$$
\Psi(rp,sp,tp,up)=\Omega(rp,sp,tp,up)
$$
provided that either one of $r,s,t,u$ is zero, or $(2a-b)r=s+t+u-1$.
\end{lemma}
\begin{proof}
First, assume that $(2a-b)r=s+t+u-1$. Let $n=p-1-a-b$. Since $p-1\leq3a+b$, we have $n\leq M$. According to (5.1),
\begin{align*}
&\Psi(rp,sp,tp,up)\\
&=\frac{(1-a+arp)_n(1+b+arp-brp-sp)_n(1+b+arp-brp-tp)_n(1+a+arp-sp-tp)_n}{(1-a+b+arp-brp)_n(1+arp-sp)_n(1+arp-tp)_n(1+a+b+arp-brp-sp-tp)_n}.
\end{align*}
It is easy to verify that
\begin{align*}
&\frac{(1+b+arp-brp-sp)_n(1+b+arp-brp-tp)_n(1+a+arp-sp-tp)_n}{(1+arp-sp)_n(1+arp-tp)_n(1+a+b+arp-brp-sp-tp)_n}\\
&\;=\frac{\Gamma_p(1+b+arp-brp-sp+n)\Gamma_p(1+b+arp-brp-tp+n)\Gamma_p(1+a+arp-sp-tp+n)}{\Gamma_p(1+b+arp-brp-sp)\Gamma_p(1+b+arp-brp-tp)\Gamma_p(1+a+arp-sp-tp)}\\
&\quad\cdot\frac{\Gamma_p(1+arp-sp)\Gamma_p(1+arp-tp)\Gamma_p(1+a+b+arp-brp-sp-tp)}{\Gamma_p(1+arp-sp+n)\Gamma_p(1+arp-tp+n)\Gamma_p(1+a+b+arp-brp-sp-tp+n)}.
\end{align*}
Furthermore, since $2a+b\leq p-1$, we also have $n\geq a$. So
\begin{align*}
\frac{(1-a+arp)_n}{(1-a+b+arp-brp)_n}&=\frac{arp\cdot\Gamma_p(1-a+arp+n)\cdot\Gamma_p(1-a+b+arp-brp)}{\Gamma_p(1-a+arp)\cdot(arp-brp)\cdot\Gamma_p(1-a+b+arp-brp+n)}\\
&=\frac{a}{a-b}\cdot\frac{\Gamma_p(1-a+arp+n)\Gamma_p(1-a+b+arp-brp)}{\Gamma_p(1-a+arp)\Gamma_p(1-a+b+arp-brp+n)}.
\end{align*}
It follows that $\Psi(rp,sp,tp,up)=\Omega(rp,sp,tp,up)$.

Next, assume that $r=0$. Without loss of generality, we may suppose that $s$ is a positive integer and $tp,up\in\Q\backslash\Z$. Note that $(-a+ax)_k(-b+bx)_k$ is divisible by $x^2$ for each $k\geq a$ and the denominator of $\Psi(x,sp,tp,up)$ has a non-zero constant term. Since $-a-sp$ is a negative integer,
\begin{align*}
&\Psi(0,sp,tp,up)=\lim_{x\rightarrow0}\Psi(x,sp,tp,up)=\lim_{x\rightarrow0}\Psi_{*}(x,sp,tp,up)\\
&\quad=\frac{a}{a-b}\cdot\frac{(1-a)_b(1+a-sp-tp)_b(1+a-sp-up)_b(1+a-tp-up)_b}{(1-sp)_b(1-tp)_b(1-up)_b(1+2a-sp-tp-up)_b}.
\end{align*}
Clearly,
\begin{align*}
&\frac{(1-a)_b(1+a-sp-tp)_b(1+a-sp-up)_b(1+a-tp-up)_b}{(1-sp)_b(1-tp)_b(1-up)_b(1+2a-sp-tp-up)_b}\\
&\;=\frac{\Gamma_p(1-a+b)\Gamma_p(1+a+b-sp-tp)\Gamma_p(1+a+b-sp-up)\Gamma_p(1+a+b-tp-up)}{\Gamma_p(1-a)\Gamma_p(1+a-sp-tp)\Gamma_p(1+a-sp-up)\Gamma_p(1+a-tp-up)}\\
&\quad\cdot\frac{\Gamma_p(1-sp)\Gamma_p(1-tp)\Gamma_p(1-up)\Gamma_p(1+2a-sp-tp-up)}{\Gamma_p(1+b-sp)\Gamma_p(1+b-tp)\Gamma_p(1+b-up)\Gamma_p(1+2a+b-sp-tp-up)}.
\end{align*}
This concludes that $\Psi(rp,sp,tp,up)=\Omega(rp,sp,tp,up)$.

Now we assume that $s=0$. Assume that $rp,tp,up\in\Q\backslash\Z$. Then
\begin{align*}
&\Psi(rp,0,tp,up)\\
&=\frac{(1-a+arp)_a(1+b+arp-brp-tp)_a(1+b+arp-brp-up)_a(1+a+arp-tp-up)_a}{(1-a+b+arp-brp)_a(1+arp-tp)_a(1+arp-up)_a(1+a+b+arp-brp-tp-up)_a}.
\end{align*}
It is easy to check that
\begin{align*}
&\frac{(1+b+arp-brp-tp)_a(1+b+arp-brp-up)_a(1+a+arp-tp-up)_a}{(1+arp-tp)_a(1+arp-up)_a(1+a+b+arp-brp-tp-up)_a}\\
&=\frac{\Gamma_p(1+a+b+arp-brp-tp)\Gamma_p(1+a+b+arp-brp-up)\Gamma_p(1+2a+arp-tp-up)}{\Gamma_p(1+b+arp-brp-tp)\Gamma_p(1+b+arp-brp-up)\Gamma_p(1+a+arp-tp-up)}\\
&\;\cdot\frac{\Gamma_p(1+arp-tp)\Gamma_p(1+arp-up)\Gamma_p(1+a+b+arp-brp-tp-up)}{\Gamma_p(1+a+arp-tp)\Gamma_p(1+a+arp-up)\Gamma_p(1+2a+b+arp-brp-tp-up)}.
\end{align*}
Also, we have
$$
\frac{(1-a+arp)_a}{(1-a+b+arp-brp)_a}=\frac{a}{a-b}\cdot\frac{\Gamma_p(1+arp)\Gamma_p(1-a+b+arp-brp)}{\Gamma_p(1-a+arp)\Gamma_p(1+b+arp-brp)}.
$$
Thus $\Psi(rp,sp,tp,up)=\Omega(rp,sp,tp,up)$ also holds.

Finally, assume that $u=0$. Then
\begin{align*}
&\Psi(rp,sp,tp,0)\\
&=\frac{(1-a+arp)_a(1+b+arp-brp-sp)_a(1+b+arp-brp-tp)_a(1+a+arp-sp-tp)_a}{(1-a+b+arp-brp)_a(1+arp-sp)_a(1+arp-tp)_a(1+a+b+arp-brp-sp-tp)_a}.
\end{align*}
Similarly to the above, we may check that
\begin{align*}
&\frac{(1+b+arp-brp-sp)_a(1+b+arp-brp-tp)_a(1+a+arp-sp-tp)_a}{(1+arp-sp)_a(1+arp-tp)_a(1+a+b+arp-brp-sp-tp)_a}\\
&=\frac{\Gamma_p(1+a+b+arp-brp-sp)\Gamma_p(1+a+b+arp-brp-tp)\Gamma_p(1+2a+arp-sp-tp)}{\Gamma_p(1+b+arp-brp-sp)\Gamma_p(1+b+arp-brp-tp)\Gamma_p(1+a+arp-sp-tp)}\\
&\;\cdot\frac{\Gamma_p(1+arp-sp)\Gamma_p(1+arp-tp)\Gamma_p(1+a+b+arp-brp-sp-tp)}{\Gamma_p(1+a+arp-sp)\Gamma_p(1+a+arp-tp)\Gamma_p(1+2a+b+arp-brp-sp-tp)}
\end{align*}
and
$$
\frac{(1-a+arp)_a}{(1-a+b+arp-brp)_a}=\frac{a}{a-b}\cdot\frac{\Gamma_p(1+arp)\Gamma_p(1-a+b+arp-brp)}{\Gamma_p(1-a+arp)\Gamma_p(1+b+arp-brp)}.
$$
Thus $\Psi(rp,sp,tp,up)=\Omega(rp,sp,tp,up)$ holds again.

Combining all the above equalities we are done.
\end{proof}
Furthermore, we also have
\begin{theorem}
Let $p\geq 3$ be prime and $\alpha\in\Z_p^\times$ with $\langle-\alpha\rangle_p<p/3$.
Then
\begin{align}\label{alphabetagammadeltan7F6Be}
&{}_7F_6\bigg[\begin{matrix}\alpha&\alpha&\alpha&\alpha&\alpha&1+\frac12\alpha&1-2\alpha\\
&1&1&1&1&\frac12\alpha&3\alpha\end{matrix}\bigg|\,1\bigg]_{M}\notag\\
&\;\equiv p\alpha_p^*\cdot\frac{\Gamma_p(1-2\alpha)^4}{\Gamma_p(1-\alpha)^6\Gamma_p(1+\alpha)\Gamma_p(1-3\alpha)}\pmod{p^6},
\end{align}
where
$M=3\langle-\alpha\rangle_p.
$
\end{theorem}

\section{More congruences arising from hypergeometric identities}
\setcounter{equation}{0}\setcounter{theorem}{0}
\setcounter{lemma}{0}\setcounter{corollary}{0}

In \cite{MP17}, Pan and Mao obtained the $\mod p^2$ analogues of many hypergeometric identities. In this section, we shall give the $\mod p^3$ extensions of some results of Pan and Mao in special cases.

First, a consequence of Dougall's formula \eqref{Dougall7F6} is (cf. \cite[Corollary 3.5.2]{AAR99})
\begin{align}\label{Dougall5F4}
&{}_5F_4\bigg[\begin{matrix}\alpha&1+\frac12\alpha&\beta&\gamma&\delta\\
&\frac12\alpha&\alpha-\beta+1&\alpha-\gamma+1&\alpha-\delta+1\end{matrix}\bigg|\,1\bigg]\notag\\
&\qquad=\frac{\Gamma(\alpha-\beta+1)\Gamma(\alpha-\gamma+1)\Gamma(\alpha-\delta+1)\Gamma(\alpha-\beta-\gamma-\delta+1)}{\Gamma(\alpha+1)\Gamma(\alpha-\beta-\gamma+1)\Gamma(\alpha-\beta-\delta+1)\Gamma(\alpha-\gamma-\delta+1)}.
\end{align}
Setting $\gamma=\delta=1$ in \eqref{Dougall5F4}, we get
\begin{align}\label{Dougall5F411}
{}_5F_4\bigg[\begin{matrix}\alpha&1+\frac12\alpha&\alpha&\alpha&\beta\\
&\frac12\alpha&1&1&\alpha-\beta+1\end{matrix}\bigg|\,1\bigg]=\frac{\Gamma(\alpha-\beta+1)\Gamma(1-\alpha-\beta)}{\Gamma(1+\alpha)\Gamma(1-\alpha)\Gamma(1-\beta)^2}.
\end{align}
The following two theorems are the $\mod p^3$ analogues of \eqref{Dougall5F411}.
\begin{theorem}\label{5F4modp3a}
Let $\alpha,\beta\in\Z_p$ with $\langle-\beta\rangle_p<\langle-\alpha\rangle_p$ and $\langle-\alpha\rangle_p/\alpha=\langle-\beta\rangle_p/\beta$. If $\langle-\alpha\rangle_p+\langle-\beta\rangle_p<p$ and $p^2\nmid(\alpha-\beta+1)_{p-1}$, then
\begin{align}&{}_5F_4\bigg[\begin{matrix}\alpha&1+\frac{\alpha}{2}&\alpha&\alpha&\beta\\ &\frac{\alpha}{2}&1&1&1+\alpha-\beta\end{matrix}\bigg|\,1\bigg]_{p-1}\notag\\
&\qquad\equiv\frac{\alpha}{\alpha-\beta}\cdot\frac{\Gamma_p(1+\alpha-\beta)\Gamma_p(1-\alpha-\beta)}{\Gamma_p(1+\alpha)\Gamma_p(1-\alpha)\Gamma_p(1-\beta)^2}\pmod{p^3}.
\end{align}
\end{theorem}
\begin{theorem}\label{5F4modp3b}
Let $\alpha,\beta\in\Z_p$ with $\langle-\alpha\rangle_p\leq\langle-\beta\rangle_p$ and $p^2\nmid (\alpha-\beta+1)_{p-1}$.
If $\langle-\alpha\rangle_p+\langle-\beta\rangle_p\geq p$, then
\begin{align}&{}_5F_4\bigg[\begin{matrix}\alpha&1+\frac{\alpha}{2}&\alpha&\alpha&\beta\\ &\frac{\alpha}{2}&1&1&1+\alpha-\beta\end{matrix}\bigg|\,1\bigg]_{p-1}\notag\\
&\qquad\equiv p^2\alpha_p^*(1-\alpha_p^*-\beta_p^*)\cdot\frac{\Gamma_p(1+\alpha-\beta)\Gamma_p(1-\alpha-\beta)}{\Gamma_p(1+\alpha)\Gamma_p(1-\alpha)\Gamma_p(1-\beta)^2}\pmod{p^3}.
\end{align}
If $\langle-\alpha\rangle_p+\langle-\beta\rangle_p<p$, then
\begin{align}
&{}_5F_4\bigg[\begin{matrix}\alpha&1+\frac{\alpha}{2}&\alpha&\alpha&\beta\\ &\frac{\alpha}{2}&1&1&1+\alpha-\beta\end{matrix}\bigg|\,1\bigg]_{p-1}\notag\\
&\qquad\equiv p\alpha_p^*\cdot\frac{\Gamma_p(1+\alpha-\beta)\Gamma_p(1-\alpha-\beta)}{\Gamma_p(1+\alpha)\Gamma_p(1-\alpha)\Gamma_p(1-\beta)^2}\pmod{p^3}.
\end{align}
\end{theorem}

Furthermore, setting $\beta=\alpha$ in Theorem \ref{5F4modp3b}, we can get a stronger result.
\begin{theorem}\label{5F4analog}
Let $\alpha\in\Z_p^{\times}$. If $\langle-\alpha\rangle_p\geq(p+1)/2$, then
\begin{align}
{}_5F_4\bigg[\begin{matrix}\alpha&1+\frac{\alpha}2&\alpha&\alpha&\alpha\\&\frac{\alpha}2&1&1&1\end{matrix}\bigg|\ 1\bigg]_{p-1}\equiv p^2\alpha_p^*(1-2\alpha_p^*)\cdot\frac{\Gamma_p(1-2\alpha)}{\Gamma_p(1+\alpha)\Gamma_p(1-\alpha)^3}\pmod{p^4}.
\end{align}
If $\langle-\alpha\rangle_p\leq(p+1)/2$, then
\begin{align}
{}_5F_4\bigg[\begin{matrix}\alpha&1+\frac{\alpha}2&\alpha&\alpha&\alpha\\&\frac{\alpha}2&1&1&1\end{matrix}\bigg|\ 1\bigg]_{p-1}\equiv p\alpha_p^*\cdot\frac{\Gamma_p(1-2\alpha)}{\Gamma_p(1+\alpha)\Gamma_p(1-\alpha)^3}\pmod{p^4}.
\end{align}
\end{theorem}

In fact, Dougall's formula \eqref{Dougall7F6} can be deduced from  Whipple's ${}_7F_6$ transformation \cite[Theorem 3.4.5]{AAR99}:
\begin{align}\label{Whippletransformation7F6}
&{}_7F_6\bigg[\begin{matrix}\alpha&1+\frac12\alpha&\beta&\gamma&\delta&\epsilon&\rho\\
&\frac12\alpha&\alpha-\beta+1&\alpha-\gamma+1&\alpha-\delta+1&\alpha-\epsilon+1&\alpha-\rho+1\end{matrix}\bigg|\,1\bigg]\notag\\
&\quad=\frac{\Gamma(\alpha-\beta+1)\Gamma(\alpha-\gamma+1)\Gamma(\alpha-\delta+1)\Gamma(\alpha-\beta-\gamma-\delta+1)}{\Gamma(\alpha+1)\Gamma(\alpha-\beta-\gamma+1)\Gamma(\alpha-\beta-\delta+1)\Gamma(\alpha-\gamma-\delta+1)}\notag\\
&\qquad\cdot{}_4F_3\bigg[\begin{matrix} \alpha-\epsilon-\rho+1&\beta&\gamma&\delta\\ &\beta+\gamma+\delta-\alpha&\alpha-\epsilon+1&\alpha-\rho+1\end{matrix}\bigg|\,1\bigg].
\end{align}
By substituting $\epsilon=\rho=\alpha$ in \eqref{Whippletransformation7F6}, we get
\begin{align}\label{7F6whipple}
&{}_7F_6\bigg[\begin{matrix}\alpha&1+\frac{1}{2}\alpha&\alpha&\alpha&\beta&\gamma&\delta\\&\frac12\alpha&1&1&1+\alpha-\beta&1+\alpha-\gamma&1+\alpha-\delta\end{matrix}\bigg|\ 1\bigg]\notag\\
&\quad=\frac{\Gamma(\alpha-\beta+1)\Gamma(\alpha-\gamma+1)\Gamma(\alpha-\delta+1)\Gamma(\alpha-\beta-\gamma-\delta+1)}{\Gamma(\alpha+1)\Gamma(\alpha-\beta-\gamma+1)\Gamma(\alpha-\beta-\delta+1)\Gamma(\alpha-\gamma-\delta+1)}\notag\\
&\qquad\cdot{}_4F_3\bigg[\begin{matrix}1-\alpha&\beta&\gamma&\delta\\&1&1&\beta+\gamma+\delta-\alpha\end{matrix}\bigg|\ 1\bigg].
\end{align}
There are three $\mod p^3$ analogues of \eqref{7F6whipple}.
\begin{theorem}\label{7F6analog(1)}
Let $\alpha,\beta,\gamma,\delta\in\Z_p$. Suppose that

\medskip\noindent
{\rm(i)} $\langle-\beta\rangle_p<\langle-\alpha\rangle_p\leq\min\{\langle-\gamma\rangle_p,\langle-\delta\rangle_p\}$;

\medskip\noindent
{\rm(ii)} $p+\langle-\alpha\rangle_p>\langle-\beta\rangle_p+\langle-\gamma\rangle_p+\langle-\delta\rangle_p$;

\medskip\noindent
{\rm(iii)} $\langle-\alpha\rangle_p/\alpha=\langle-\beta\rangle_p/\beta$;

\medskip\noindent
{\rm(iv)} $(\alpha-\beta+1)_{p-1},(\alpha-\gamma+1)_{p-1},(\alpha-\delta+1)_{p-1},(\beta+\gamma+\delta-\alpha)_{p-1}$ are not divisible by $p^2$.

\medskip\noindent
Then
\begin{align}\label{7F6analog(1)c}
&{}_7F_6\bigg[\begin{matrix}\alpha&1+\frac{1}{2}\alpha&\alpha&\alpha&\beta&\gamma&\delta\\&\frac12\alpha&1&1&1+\alpha-\beta&1+\alpha-\gamma&1+\alpha-\delta\end{matrix}\bigg|\ 1\bigg]_{p-1}\notag\\
&\quad\equiv\frac{\alpha}{\alpha-\beta}\cdot\frac{\Gamma_p(\alpha-\beta+1)\Gamma_p(\alpha-\gamma+1)\Gamma_p(\alpha-\delta+1)\Gamma_p(\alpha-\beta-\gamma-\delta+1)}{\Gamma_p(\alpha+1)\Gamma_p(\alpha-\beta-\gamma+1)\Gamma_p(\alpha-\beta-\delta+1)\Gamma_p(\alpha-\gamma-\delta+1)}\notag\\
&\qquad\cdot{}_4F_3\bigg[\begin{matrix}1-\alpha&\beta&\gamma&\delta\\&1&1&\beta+\gamma+\delta-\alpha\end{matrix}\bigg|\ 1\bigg]_{p-1}\pmod{p^3}.
\end{align}
\end{theorem}
\begin{theorem}\label{7F6analog(2)}
Let $\alpha,\beta,\gamma,\delta\in\Z_p$. Suppose that

\medskip\noindent{\rm(i)} $\langle-\alpha\rangle_p\leq\min\{\langle-\beta\rangle_p,\langle-\gamma\rangle_p,\langle-\delta\rangle_p\}$;

\medskip\noindent{\rm(ii)} $p+\langle-\alpha\rangle_p>\max\{\langle-\beta\rangle_p+\langle-\gamma\rangle_p,\langle-\beta\rangle_p+\langle-\delta\rangle_p,\langle-\gamma\rangle_p+\langle-\delta\rangle_p\}$;

\medskip\noindent{\rm(iii)} $2p-1\leq\langle-\beta\rangle_p+\langle-\gamma\rangle_p+\langle-\delta\rangle_p$;

\medskip\noindent{\rm(iv)} $(\alpha-\beta+1)_{p-1},(\alpha-\gamma+1)_{p-1},(\alpha-\delta+1)_{p-1}$ are not divisible by $p^2$.

\medskip\noindent Then
\begin{align}
&{}_7F_6\bigg[\begin{matrix}\alpha&1+\frac{1}{2}\alpha&\alpha&\alpha&\beta&\gamma&\delta\\&\frac12\alpha&1&1&1+\alpha-\beta&1+\alpha-\gamma&1+\alpha-\delta\end{matrix}\bigg|\ 1\bigg]_{p-1}\notag\\
&\quad\equiv p^2\alpha_p^*(1+\alpha_p^*-\beta_p^*-\gamma_p^*-\delta_p^*)\notag\\
&\qquad\cdot\frac{\Gamma_p(\alpha-\beta+1)\Gamma_p(\alpha-\gamma+1)\Gamma_p(\alpha-\delta+1)\Gamma_p(\alpha-\beta-\gamma-\delta+1)}{\Gamma_p(\alpha+1)\Gamma_p(\alpha-\beta-\gamma+1)\Gamma_p(\alpha-\beta-\delta+1)\Gamma_p(\alpha-\gamma-\delta+1)}\notag\\
&\qquad\cdot{}_4F_3\bigg[\begin{matrix}1-\alpha&\beta&\gamma&\delta\\&1&1&\beta+\gamma+\delta-\alpha\end{matrix}\bigg|\ 1\bigg]_{p-1}\pmod{p^3}.
\end{align}
\end{theorem}
\begin{theorem}\label{7F6analog(3)}
Let $\alpha,\beta,\gamma,\delta\in\Z_p$. Suppose that

\medskip\noindent{\rm (i)} $\langle-\alpha\rangle_p\leq\min\{\langle-\beta\rangle_p,\langle-\gamma\rangle_p,\langle-\delta\rangle_p\}$;

\medskip\noindent {\rm (ii)} $p+\langle-\alpha\rangle_p>\langle-\beta\rangle_p+\langle-\gamma\rangle_p+\langle-\delta\rangle_p$;

\medskip\noindent {\rm (iii)} $(\alpha-\beta+1)_{p-1},(\alpha-\gamma+1)_{p-1},(\alpha-\delta+1)_{p-1},(\beta+\gamma+\delta-\alpha)_{p-1}$ are not divisible by $p^2$.

\medskip\noindent Then
\begin{align}
&{}_7F_6\bigg[\begin{matrix}\alpha&1+\frac{1}{2}\alpha&\alpha&\alpha&\beta&\gamma&\delta\\&\frac12\alpha&1&1&1+\alpha-\beta&1+\alpha-\gamma&1+\alpha-\delta\end{matrix}\bigg|\ 1\bigg]_{p-1}\notag\\
&\quad\equiv p\alpha_p^*\cdot\frac{\Gamma_p(\alpha-\beta+1)\Gamma_p(\alpha-\gamma+1)\Gamma_p(\alpha-\delta+1)\Gamma_p(\alpha-\beta-\gamma-\delta+1)}{\Gamma_p(\alpha+1)\Gamma_p(\alpha-\beta-\gamma+1)\Gamma_p(\alpha-\beta-\delta+1)\Gamma_p(\alpha-\gamma-\delta+1)}\notag\\
&\qquad\cdot{}_4F_3\bigg[\begin{matrix}1-\alpha&\beta&\gamma&\delta\\&1&1&\beta+\gamma+\delta-\alpha\end{matrix}\bigg|\ 1\bigg]_{p-1}\pmod{p^3}.
\end{align}
\end{theorem}

Theorems \ref{5F4modp3a}-\ref{7F6analog(3)} can be proved in a similar way with the help of Theorem \ref{main}. Hence the detailed proofs of those theorems will not be given here. For example, under the assumptions of Theorem \ref{7F6analog(1)}, let
\begin{align*}
&\Psi(x,y,z):={}_7F_6\bigg[\begin{matrix}{\mathfrak a}&1+\frac12{\mathfrak a}&{\mathfrak a}&{\mathfrak a}&{\mathfrak b}&{\mathfrak c}&{\mathfrak d}\\
&\frac12{\mathfrak a}&1&1&{\mathfrak a}-{\mathfrak b}+1&{\mathfrak a}-{\mathfrak c}+1&{\mathfrak a}-{\mathfrak d}+1\end{matrix}\bigg|\,1\bigg]_{p-1}\\
&\qquad-\frac{{\mathfrak a}}{{\mathfrak a}-{\mathfrak b}}\cdot\frac{\Gamma_p({\mathfrak a}-{\mathfrak b}+1)\Gamma_p({\mathfrak a}-{\mathfrak c}+1)\Gamma_p({\mathfrak a}-{\mathfrak d}+1)\Gamma_p({\mathfrak a}-{\mathfrak b}-{\mathfrak c}-{\mathfrak d}+1)}{\Gamma_p({\mathfrak a}+1)\Gamma_p({\mathfrak a}-{\mathfrak b}-{\mathfrak c}+1)\Gamma_p({\mathfrak a}-{\mathfrak b}-{\mathfrak d}+1)\Gamma_p({\mathfrak a}-{\mathfrak c}-{\mathfrak d}+1)}\\
&\qquad\cdot{}_4F_3\bigg[\begin{matrix} 1-{\mathfrak a}&{\mathfrak b}&{\mathfrak c}&{\mathfrak d}\\ &1&1&{\mathfrak b}+{\mathfrak c}+{\mathfrak d}-{\mathfrak a}\end{matrix}\bigg|\,1\bigg]_{p-1},
\end{align*}
where ${\mathfrak a}=\langle-\alpha\rangle(-1+x)$, ${\mathfrak b}=\langle-\beta\rangle(-1+x)$,
${\mathfrak c}=-\langle-\gamma\rangle+y$ and ${\mathfrak d}=-\langle-\delta\rangle+z$. Then (\ref{7F6analog(1)c}) immediately follows from
$$
\Psi(0,tp,rp)=\Psi(sp,0,rp)=\Psi(sp,tp,0)=0.
$$

\section{The congruences involving harmonic numbers}
\setcounter{equation}{0}\setcounter{theorem}{0}
\setcounter{lemma}{0}\setcounter{corollary}{0}

The following lemma will help us to find some congruences involving harmonic numbers.

\begin{lemma}\label{mainder}
Under the assumptions of Theorem \ref{main}, additionally assume that $\Psi(x_1,\ldots,x_n)$ has the Taylor expansion of order $r-1$
$$
\Psi(t_1p,\ldots,t_np)\equiv \Psi(0,\ldots,0)
+\sum_{\substack{k_1,\ldots,k_n\geq 0\\
1\leq k_1+\cdots+k_n\leq r-1}}\frac{A_{k_1,\ldots,k_n}}{k_1!\cdots k_n!}
\prod_{i=1}^n(t_ip)^{k_i}\pmod{p^{r}}.
$$
Then we have
$$
A_{k_1,\ldots,k_n}\equiv0\pmod{p^{r-k_1-\cdots-k_n}}
$$
for any $k_1,\ldots,k_{n}\geq 0$ with $1\leq k_1+\cdots+k_{n}\leq r-1$.
\end{lemma}
\begin{proof}
We proceed by induction on $n$. The case $n=1$ follows from Theorem \ref{psiaiprt}. Assume that $n>1$ and that the statement holds for any smaller value of $n$.
According to Theorem \ref{main}, we know that
$$
\Psi(t_1p,\ldots,t_np)\equiv0\pmod{p^r}
$$
for each $t_1,\ldots,t_n\in\Z_p$. Let
$$
\Omega(x_1,\ldots,x_n):=\Psi(0,\ldots,0)
+\sum_{\substack{k_1,\ldots,k_n\geq 0\\
1\leq k_1+\cdots+k_n\leq r-1}}\frac{A_{k_1,\ldots,k_n}}{k_1!\cdots k_n!}
\prod_{i=1}^nx_i^{k_i}.
$$
Moreover, we also have
$$
\Omega(t_1p,\ldots,t_np)\equiv0\pmod{p^r}.
$$
For each $a\in\Z_p$, let
$$
\Phi_a(x_1,\ldots,x_{n-1}):=\Omega(x_1,\ldots,x_{n-1},a).
$$
Then
\begin{align*}
&\Phi_a(t_1p,\ldots,t_{n-1}p)\\
&\quad=
\sum_{h=0}^{r-1}\frac{A_{0,\ldots,0,h}}{h!}\cdot(ap)^h
+\sum_{h=0}^{r-2}\frac{(ap)^h}{h!}\sum_{\substack{k_1,\ldots,k_{n-1}\geq 0\\
1\leq k_1+\cdots+k_{n-1}\leq r-1-h}}\frac{A_{k_1,\ldots,k_{n-1},h}}{k_1!\cdots k_{n-1}!}
\prod_{i=1}^{n-1}(t_ip)^{k_i},
\end{align*}
where we set $A_{0,\ldots,0}=\Psi(0,\ldots,0)$.
From the induction hypothesis we find that
\begin{equation}\label{Akihmodpr}
\sum_{h=0}^{r-1-k_1-\cdots-k_{n-1}}\frac{A_{k_1,\ldots,k_{n-1},h}}{h!}\cdot(ap)^h
\equiv 0\pmod{p^{r-k_1-\cdots-k_{n-1}}}
\end{equation}
for any $k_1,\ldots,k_{n-1}\geq 0$ with $k_1+\cdots+k_{n-1}\leq r-1$.
Since \eqref{Akihmodpr} is valid for each $a\in\Z_p$, by Theorem \ref{psiaiprt}, we finally obtain
$$
A_{k_1,\ldots,k_{n-1},h}\equiv0\pmod{p^{r-k_1-\cdots-k_{n-1}-h}}
$$
for each $0\leq h\leq r-1-k_1-\cdots-k_{n-1}$.
\end{proof}

Let us see an application of Lemma \ref{mainder}.
Recall that
$$
H_n(p)=\sum_{\substack{1\leq k\leq n\\ p\nmid k}}\frac1k.
$$

\begin{theorem}\label{alpha3harmonic}
Suppose that $p>3$ is prime and $\alpha\in\Z_p^\times$.
Let
$$g_p(\alpha)=\frac{\Gamma_p(1+\frac12\alpha)\Gamma_p(1-\frac32\alpha)}{\Gamma_p(1+\alpha)\Gamma_p(1-\alpha)\Gamma_p(1-\frac12\alpha)^2}$$
and
$$h_p(\alpha)=H_{\langle-\alpha\rangle_{p^2}}(p)
+\frac{1}{2}H_{\langle-\alpha/2\rangle_{p^2}}(p)
-\frac{1}{2}H_{\langle-3\alpha/2\rangle_{p^2}}(p).
$$
Then the following congruence holds modulo $p^2$,
$$
\sum_{k=1}^{p-1}\frac{(\alpha)_k^3}{(1)_k^3}\cdot H_k\equiv
\begin{cases}
\displaystyle 2g_p(\alpha)h_p(\alpha)&\text{if $\langle-\alpha\rangle_p$ is even and $\langle-\alpha\rangle_p<2p/3$,}
\vspace{2mm}\\
\displaystyle g_p(\alpha)(-1+p(2-3\alpha_p^*)h_p(\alpha))&\text{if $\langle-\alpha\rangle_p$ is even and $\langle-\alpha\rangle_p\geq 2p/3$,}
\vspace{2mm}\\
\displaystyle g_p(\alpha)\left(-\frac{1}{3}+p\alpha_p^*h_p(\alpha)\right)&\text{if $\langle-\alpha\rangle_p$ is odd and $\langle-\alpha\rangle_p<p/3$,}
\vspace{2mm}\\
\displaystyle -\frac{pg_p(\alpha)}{6}&\text{if $\langle-\alpha\rangle_p$ is odd and $\langle-\alpha\rangle_p\geq p/3$.}
\end{cases}
$$
\end{theorem}
\begin{proof} Here we only show the first case where $\langle-\alpha\rangle_p$ is even and $\langle-\alpha\rangle_p<2p/3$, since the proofs of the other cases are very similar.

Suppose that a function $f(x_1,\ldots,x_n)$ has the Taylor series of at $(a_1,\ldots,a_n)$
$$
f(t_1p,\ldots,t_np)=f(a_1,\ldots,a_n)
+\sum_{\substack{k_1,\ldots,k_n\geq 0\\
k_1+\cdots+k_n\geq 1}}\frac{A_{k_1,\ldots,k_n}(a_1,\ldots,a_n)}{k_1!\cdots k_n!}
\prod_{i=1}^n(t_ip-a_i)^{k_i}.
$$
Write
$$
\frac{\partial_p^{k_1+\cdots+k_n}f}{\partial_p x_1^{k_1}\cdots\partial_p x_n^{k_n}}(a_1,\ldots,a_n)=A_{k_1,\ldots,k_n}(a_1,\ldots,a_n).
$$
In particular, for any $1\leq r\leq p-1$, if $f$ has a strong Taylor expansion of order $r$ at $(a_1,\ldots,a_n)$, then
\begin{align*}
&f(t_1p,\ldots,t_np)-f(a_1,\ldots,a_n)\\
&\quad\equiv\sum_{\substack{k_1,\ldots,k_n\geq 0\\
1\leq k_1+\cdots+k_n\leq r}}\frac{1}{k_1!\cdots k_n!}\cdot\frac{\partial_p^{k_1+\cdots+k_n}f(a_1,\ldots,a_n)}{\partial_p x_1^{k_1}\cdots\partial_p x_n^{k_n}}
\cdot\prod_{i=1}^n(t_ip-a_i)^{k_i}\pmod{p^{r+1}}.
\end{align*}

Let
$$
\Psi(x,y):={}_3F_2\bigg[\begin{matrix}\alpha+x & \alpha+y & \alpha\\ & 1+x-y & 1+x\end{matrix}\bigg|\ 1\bigg]_{p-1}.
$$
Then
\begin{align*}
\frac{\partial_p\Psi(0,0)}{\partial_p x}&\equiv
\frac{\partial\Psi(0,0)}{\partial x}=
\sum_{k=0}^{p-1}\frac{\partial}{\partial x}\bigg(\frac{(\alpha+x)_k(\alpha)_k^2}{(1+x)_k^2(1)_k}\bigg)\bigg|_{x=0}\\
&=\sum_{k=0}^{p-1}\frac{(\alpha)_k^3}{(1)_k^3}\bigg(\sum_{j=0}^{k-1}\frac{1}{\alpha+j}-2H_k\bigg)\pmod{p^2}.
\end{align*}
Similarly,
\begin{align*}
\frac{\partial_p\Psi(0,0)}{\partial_p y}\equiv
\frac{\partial\Psi(0,0)}{\partial y}=
\sum_{k=0}^{p-1}\frac{(\alpha)_k^3}{(1)_k^3}\bigg(\sum_{j=0}^{k-1}\frac{1}{\alpha+j}+H_k\bigg)\pmod{p^2}.
\end{align*}

On the other hand, let
$$
\Omega(x,y):=\frac{2\Gamma_p(1+\frac{1}{2}\alpha+\frac{1}{2}x)\Gamma_p(1+x-y)\Gamma_p(1+x)\Gamma_p(1-\frac{3}{2}\alpha+\frac{1}{2}x-y)}{\Gamma_p(1+\alpha+x)\Gamma_p(1-\frac{1}{2}\alpha+\frac{1}{2}x-y)\Gamma_p(1-\frac{1}{2}\alpha+\frac{1}{2}x)\Gamma_p(1-\alpha+x-y)}.
$$
Let $G_n(x)$ be the function appearing in \eqref{Gammaptptaylor} and $\fH_n^{(s)}(p)$ be the one defined in \eqref{fHnspdef}.
Then by \eqref{Gkalpha2},
\begin{align*}
\frac{\partial_p\Omega(0,0)}{\partial_p x}&
\equiv\frac12G_1\bigg(1+\frac{1}{2}\alpha\bigg)+2G_1(1)+\frac12G_1\bigg(1-\frac{3}{2}\alpha\bigg)\\
&\qquad-G_1(1+\alpha)-G_1(1-\alpha)-G_1\bigg(1-\frac{1}{2}\alpha\bigg)\\
&\equiv\frac{\fH_{\langle \frac{1}{2}\alpha\rangle_{p^2}}^{(1)}(p)}{2}+
\frac{\fH_{\langle -\frac{3}{2}\alpha\rangle_{p^2}}^{(1)}(p)}{2}-
\fH_{\langle \alpha\rangle_{p^2}}^{(1)}(p)-\fH_{\langle-\alpha\rangle_{p^2}}^{(1)}(p)-\fH_{\langle -\frac12\alpha\rangle_{p^2}}^{(1)}(p)
\pmod{p^2}.
\end{align*}
Similarly,
\begin{align*}
\frac{\partial_p\Omega(0,0)}{\partial_p y}\equiv&-G_1(1)-G_1\bigg(1-\frac{3}{2}\alpha\bigg)+G_1(1-\alpha)+G_1\bigg(1-\frac{1}{2}\alpha\bigg)\\
\equiv&-
\fH_{\langle -\frac{3}{2}\alpha\rangle_{p^2}}^{(1)}(p)+\fH_{\langle-\alpha\rangle_{p^2}}^{(1)}(p)+\fH_{\langle -\frac12\alpha\rangle_{p^2}}^{(1)}(p)
\pmod{p^2}.
\end{align*}

Note that by Lemma \ref{dixonPsirst0}, we have
$$
\Psi(rp,0)\equiv\Omega(rp,0)\pmod{p^3},\qquad
\Psi(0,sp)\equiv\Omega(0,sp)\pmod{p^3}
$$
for any $r,s\in\Z_p$. Thus by Theorem \ref{main},
$$
\Psi(rp,sp)\equiv\Omega(rp,sp)\pmod{p^3}
$$
for any $r,s\in\Z_p$.
By applying Lemma \ref{mainder}, we get
$$
\frac{\partial_p\Psi(0,0)}{\partial_p y}-\frac{\partial_p\Psi(0,0)}{\partial_p x}\equiv\frac{\partial_p\Omega(0,0)}{\partial_p y}-\frac{\partial_p\Omega(0,0)}{\partial_p x}\pmod{p^2}.
$$
Recall that $\fH_{n}^{(1)}(p)=H_n(p)$ for each $n\geq 0$. Hence
\begin{align*}
\frac{3}{\Omega(0,0)}\sum_{k=0}^{p-1}\frac{(\alpha)_k^3}{(1)_k^3}\cdot H_k
&\equiv
H_{\langle \alpha\rangle_{p^2}}(p)+2H_{\langle-\alpha\rangle_{p^2}}(p)+2H_{\langle -\frac12\alpha\rangle_{p^2}}(p)
\\
&\qquad
-\frac{1}{2}H_{\langle \frac{1}{2}\alpha\rangle_{p^2}}(p)-\frac{3}{2}H_{\langle -\frac{3}{2}\alpha\rangle_{p^2}}(p)\pmod{p^2}.
\end{align*}

By \eqref{Hnmspodd}, we have
$$
H_{\langle \alpha\rangle_{p^2}}(p)\equiv
H_{\langle-\alpha\rangle_{p^2}-1}(p)=
H_{\langle-\alpha\rangle_{p^2}}(p)+\frac{1}{\alpha}\pmod{p^2}
$$
and
$$
H_{\langle\frac12\alpha\rangle_{p^2}}(p)\equiv
H_{\langle-\frac12\alpha\rangle_{p^2}-1}(p)=
H_{\langle-\frac12\alpha\rangle_{p^2}}(p)+\frac{2}{\alpha}\pmod{p^2}.
$$
So
\begin{align*}
\frac{3}{\Omega(0,0)}\sum_{k=0}^{p-1}\frac{(\alpha)_k^3}{(1)_k^3}\cdot H_k
\equiv&
3H_{\langle-\alpha\rangle_{p^2}}(p)+\frac32H_{\langle -\frac12\alpha\rangle_{p^2}}(p)-\frac{3}{2}H_{\langle -\frac{3}{2}\alpha\rangle_{p^2}}(p)\pmod{p^2}
\end{align*}
and we are done.
\end{proof}

As special cases, by letting $\alpha=1/2$ (which appears in \cite[(5.1)]{Tauraso18}) and $\alpha=1/3$ in Theorem \ref{alpha3harmonic} we obtain the following result.

\begin{corollary}
Suppose that $p>3$ is prime and $q_p(x)=(x^{p-1}-1)/p$.
Then these two congruences hold modulo $p^2$,
$$
\sum_{k=1}^{p-1}\frac{(\frac12)_k^3}{(1)_k^3}\cdot H_k\equiv
\begin{cases}
\displaystyle 2\Gamma_p\Big(\frac14\Big)^4(2q_p(2)-pq_p(2)^2)&\text{if $p\equiv 1\pmod{4}$,}
\vspace{2mm}\\
\displaystyle -\frac{p\Gamma_p(\frac14)^4}{12}&\text{if $p\equiv 3\pmod{4}$,}
\end{cases}
$$
and
$$
\sum_{k=1}^{p-1}\frac{(\frac13)_k^3}{(1)_k^3}\cdot H_k\equiv
\begin{cases}
\displaystyle \frac{9}{8}\Gamma_p\Big(\frac16\Big)^3\Gamma_p\Big(\frac12\Big)(2q_p(3)-pq_p(3)^2)&\text{if $p\equiv 1\pmod{3}$,}
\vspace{2mm}\\
\displaystyle -\frac{p\Gamma_p(\frac16)^3\Gamma_p(\frac12)}{12}
&\text{if $p\equiv 2\pmod{3}$.}
\end{cases}
$$
\end{corollary}

\newpage

\section*{Appendix}
\setcounter{equation}{0}\setcounter{theorem}{0}
\setcounter{lemma}{0}\setcounter{corollary}{0}

\renewcommand{\theequation}{A.\arabic{equation}}
\renewcommand{\thetheorem}{A.\arabic{theorem}}
\renewcommand{\theconjecture}{A.\arabic{conjecture}}
\renewcommand{\theproposition}{A.\arabic{proposition}}
\renewcommand{\thelemma}{A.\arabic{lemma}}
\renewcommand{\thecorollary}{A.\arabic{corollary}}

Now we shall prove Lemma \ref{Hnmspl},
Lemma \ref{fHnmspl} and Theorem \ref{Gkalphat}.
\begin{lemma}\label{Hsprl} Let $p$ be prime and $r\geq 1$.
\begin{equation}\label{Hnspr}
H_{p^r}^{(s)}(p)\equiv 0\begin{cases}
\pmod{p^r},&\text{if }p-1\nmid s,\\
\pmod{p^{r-\nu_p(s)-1}},&\text{if }p-1\mid s.
\end{cases}
\end{equation}
\end{lemma}
\begin{proof}
Let $g$ be a primitive root of $p^r$.
Suppose that $p-1\nmid s$. Then
\begin{equation}\label{fHpr0pr}
\sum_{\substack{1\leq k\leq p^r\\
p\nmid k}}\frac{1}{k^s}\equiv\sum_{k=0}^{p^{r-1}(p-1)}\frac{1}{g^{ks}}=\frac{g^{p^{r-1}(p-1)s}-1}{g^s-1}\equiv0\pmod{p^r}.
\end{equation}
If $p-1\mid s$, let $d=\nu_p(s)$. There is nothing to do when $d\geq r-1$. So we may assume that
$d\leq r-2$. We must have $g^{s}\not\equiv 1\pmod{p^{d+2}}$, otherwise
$$
g^{p^{r-d-2}s}=1+\sum_{k=1}^{p^{r-d-2}}\binom{p^{r-d-2}}{k}\cdot(g^{s}-1)^k\equiv1\pmod{p^r},
$$
which contradicts the assumption that $g$ is a primitive root of $p^r$. In view of \eqref{fHpr0pr}, we also have
\begin{equation}\label{fHpr0pr2}
\sum_{\substack{1\leq k\leq p^r\\
p\nmid k}}\frac{1}{k^s}\equiv\frac{g^{p^{r-1}(p-1)s}-1}{g^s-1}\equiv 0\pmod{p^{r-d-1}}.
\end{equation}
\end{proof}
Lemma \ref{Hnmspl} easily follows from Lemma \ref{Hsprl}.
\begin{proof}[Proof of Lemma \ref{Hnmspl}]
Assume that $n>m$ and $n\equiv m\pmod{p^r}$.
Letting $h=(n-m)/p^r$, we find
\begin{align*}
H_{n}^{(s)}(p)-H_{m}^{(s)}(p)&=\sum_{j=0}^{h-1}\sum_{\substack{1\leq k\leq p^r\\
p\nmid k}}\frac{1}{(jp^r+k)^s}+\sum_{\substack{1\leq k\leq m\\p\nmid k}}\frac{1}{(hp^r+k)^s}-
\sum_{\substack{1\leq k\leq m\\p\nmid k}}\frac{1}{k^s}\\
&\equiv h\sum_{\substack{1\leq k\leq p^r\\
p\nmid k}}\frac{1}{k^s}\pmod{p^r}.
\end{align*}
Clearly \eqref{Hnspr} implies \eqref{Hnmsp}.

On the other hand, assume that $n+m+1\equiv 0\pmod{p^r}$ and
write $n+m+1=hp^r$. Since $s$ is odd, $p-1$ cannot divide $s$, and we have
\begin{align*}
H_{m}^{(s)}(p)=&=\sum_{\substack{1\leq k\leq hp^r-1-n\\ p\nmid k}}\frac1{k^s}
=\sum_{\substack{n+1\leq k\leq hp^r-1\\ p\nmid k}}\frac1{(hp^r-k)^s}\\
&\equiv-\sum_{\substack{n+1\leq k\leq hp^r\\ p\nmid k}}\frac1{k^s}=H_{n}^{(s)}(p)-
H_{hp^r}^{(s)}(p)\equiv H_{n}^{(s)}(p)\pmod{p^r},
\end{align*}
where \eqref{fHpr0pr} is used in the last step.
\end{proof}

Let us turn to $\fH_{n}^{(s)}(p)$. Recall that $\eta_s=\lfloor s/(p-1)\rfloor+\nu_p(\lfloor s/(p-1)\rfloor!)$.
\begin{lemma}
\label{fHprspl}Let $p\geq 3$ be prime and $r\geq 1$. For any $1\leq s<p^r$,
\begin{equation}\label{fHprsp}
\fH_{p^r}^{(s)}(p)\equiv 0\pmod{p^{r-\eta_s}}.
\end{equation}
\end{lemma}
\begin{proof}
For each nonnegative integer $n$, the elementary symmetric polynomial is defined as
$$
E_s(x_1,\ldots,x_n):=\sum_{\substack{1\leq j_1<j_2<\ldots<j_s\leq n}}x_{j_1}x_{j_2}\cdots x_{j_s},
$$
and  the power sum symmetric polynomial is
$$
P_s(x_1,\ldots,x_n):=x_1^s+x_2^s+\cdots+x_n^s.
$$
We set $E_s(x_1,\ldots,x_n)=0$ when $n<s$.
Then
$$
H_n^{(s)}(p)=P_s\bigg(\bigg\{\frac1{k}:\,1\leq k\leq n,\ p\nmid k\bigg\}\bigg),$$
and
$$
\fH_n^{(s)}(p)=E_s\bigg(\bigg\{\frac1{k}:\,1\leq k\leq n,\ p\nmid k\bigg\}\bigg).
$$
Applying
the classical Newton-Girard formula
\begin{equation*}
E_s(x_1,\ldots,x_n)=(-1)^s\sum_{\substack{j_1+2j_2+\cdots+sj_s=s\\
j_1,j_2,\ldots,j_s\geq 0}}\prod_{i=1}^s\frac1{j_i!}\cdot\bigg(-\frac{P_{i}(x_1,\ldots,x_n)}{i}\bigg)^{j_i},
\end{equation*}
we obtain that
\begin{equation}\label{fHprspHprsp}
\fH_{p^r}^{(s)}(p)=(-1)^s\sum_{\substack{j_1+2j_2+\cdots+sj_s=s\\
j_1,j_2,\ldots,j_s\geq 0}}\prod_{i=1}^s\frac1{j_i!}\cdot\bigg(-\frac{H_{p^r}^{(i)}(p)}{i}\bigg)^{j_i}.
\end{equation}

For each $j_1,\ldots,j_s\geq 0$ with $j_1+2j_2+\cdots+sj_s=s$, let
$$
\eta(j_1,\ldots,j_s)=\sum_{i=1}^s\big(\nu_p(j_i!)+j_i(\nu_p(i)+\delta_i(\nu_p(i)+1))\big),
$$
where $\delta_i=1$ or $0$ according to whether $p-1\mid i$ or not. We need to show that
\begin{equation}
\eta_s=\max_{\substack{j_1+2j_2+\cdots+sj_s=s\\
j_1,j_2,\ldots,j_s\geq 0}}\eta(j_1,\ldots,j_s),
\end{equation}
which evidently implies \eqref{fHprsp} by Lemma \ref{Hsprl}.

For each $2\leq i\leq s$ with $i\neq p-1$, we claim that
$$
\eta(j_1,\ldots,j_s)\leq \eta(j_1^*,\ldots,j_s^*),
$$
where
$$
j_k^*=\begin{cases}
j_{p-1}+\big\lfloor\frac{ij_i}{p-1}\big\rfloor,&\text{if }k=p-1,\\
j_{1}+\langle ij_i \rangle_{p-1},&\text{if }k=1,\\
0,&\text{if }k=i,\\
j_k,&\text{otherwise}.
\end{cases}
$$
Clearly $j_1^*+2j_2^*+\cdots+sj_s^*=s$.

\noindent If $i\geq p$ and $p-1\nmid i$, then
\begin{align*}
\eta(j_1^*,\ldots,j_s^*)-\eta(j_1,\ldots,j_s)
&=\bigg\lfloor\frac{ij_i}{p-1}\bigg\rfloor-j_i\cdot \nu_p(i)
+\nu_p\bigg(\bigg(j_{p-1}+\bigg\lfloor\frac{ij_i}{p-1}\bigg\rfloor\bigg)!\bigg)\\
&\qquad -\nu_p(j_{p-1}!)-\nu_p(j_i!)\geq 0,
\end{align*}
since $\nu_p((a+b)!)\geq\nu_p(a!)+\nu_p(b!)$ for any nonnegative integer $a,b$.

\noindent If $i\geq p$ and $p-1\mid i$, we also have
\begin{align*}
\eta(j_1^*,\ldots,j_s^*)-\eta(j_1,\ldots,j_s)
&=\bigg\lfloor\frac{ij_i}{p-1}\bigg\rfloor-j_i\bigg(2\nu_p\bigg(\frac{i}{p-1}\bigg)+1\bigg)\\
&\qquad+\nu_p\bigg(\bigg(j_{p-1}+\bigg\lfloor\frac{ij_i}{p-1}\bigg\rfloor\bigg)!\bigg)\\
&\qquad -\nu_p(j_{p-1}!)-\nu_p(j_i!)\geq0.
\end{align*}
Further, if $2\leq i<p-1$, then
\begin{align*}
\eta(j_1^*,\ldots,j_s^*)-\eta(j_1,\ldots,j_s)
&=\bigg\lfloor\frac{ij_i}{p-1}\bigg\rfloor+\nu_p\bigg(\bigg(j_{p-1}+\bigg\lfloor\frac{ij_i}{p-1}\bigg\rfloor\bigg)!\bigg)\\
&\qquad -\nu_p(j_{p-1}!)-\nu_p(j_i!)\\
&\geq\bigg\lfloor\frac{ij_i}{p-1}\bigg\rfloor-\nu_p(j_i!)\geq 0.
\end{align*}

Similarly, we have
$$
\eta(j_1,\ldots,j_s)\leq \eta(j_1^{**},\ldots,j_s^{**}),
$$
where
$$
j_k^{**}=\begin{cases}
j_{p-1}+\big\lfloor\frac{j_1}{p-1}\big\rfloor,&\text{if }k=p-1,\\
\langle j_1 \rangle_{p-1},&\text{if }k=1,\\
j_k,&\text{otherwise}.
\end{cases}
$$
In fact,
\begin{align*}
\eta(j_1^{**},\ldots,j_s^{**})-\eta(j_1,\ldots,j_s)
&=\bigg\lfloor\frac{j_1}{p-1}\bigg\rfloor+\nu_p\bigg(\bigg(j_{p-1}+\bigg\lfloor\frac{j_1}{p-1}\bigg\rfloor\bigg)!\bigg)\\
&\qquad-\nu_p(j_{p-1}!)-\nu_p(j_1!)\\
&\geq\bigg\lfloor\frac{j_1}{p-1}\bigg\rfloor-\nu_p(j_1!)\geq 0.
\end{align*}
Hence, letting
$$
j_k^{\circ}=\begin{cases}
\big\lfloor\frac{s}{p-1}\big\rfloor,&\text{if }k=p-1,\\
\langle s \rangle_{p-1},&\text{if }k=1,\\
0,&\text{otherwise},
\end{cases}
$$
we obtain that
$$
\eta(j_1,\ldots,j_s)\leq\eta(j_1^\circ,\ldots,j_s^{\circ})=
\bigg\lfloor\frac{s}{p-1}\bigg\rfloor+\nu_p\bigg(\bigg\lfloor\frac{s}{p-1}\bigg\rfloor!\bigg)=\eta_s.
$$
\end{proof}
We are ready to prove Lemma \ref{fHnmspl}.
\begin{proof}[Proof of Lemma \ref{fHnmspl}]
According to \eqref{fHprsp}, we have
$$
\prod_{\substack{1\leq k\leq p^r\\ p\nmid k}}\bigg(1+\frac xk\bigg)\equiv 1+x^{s+1}\cdot P(x)\pmod{p^{r-\eta_s}}
$$
for some polynomial $P(x)\in\Z_p[x]$.
Assume that $n-m=hp^r$ with $h\geq 1$. Then
\begin{align*}
\prod_{\substack{1\leq k\leq n\\ p\nmid k}}\bigg(1+\frac xk\bigg)\equiv&
\prod_{\substack{1\leq k\leq m\\ p\nmid k}}\bigg(1+\frac xk\bigg)\cdot\prod_{\substack{1\leq k\leq p^r\\ p\nmid k}}\bigg(1+\frac xk\bigg)^h\\
\equiv&(1+x^{s+1}\cdot P(x))^h
\prod_{\substack{1\leq k\leq m\\ p\nmid k}}\bigg(1+\frac xk\bigg)\pmod{p^{r-\eta_s}}.
\end{align*}
Comparing the coefficients of $x^s$ in the both sides of the above congruences, we get \eqref{fHnmsp}.
\end{proof}
Finally, let us complete the proof of Theorem \ref{Gkalphat}.
\begin{proof}[Proof of Theorem \ref{Gkalphat}]
Recall that
$
\Gamma_p(x)\equiv\Gamma_p(y)\pmod{p^r}
$ if $x\equiv y\pmod{p^r}$.
Let
$a=p^{r}-\langle-\alpha\rangle_{p^{r}}$.
Then $a\equiv \alpha\pmod{p^{r}}$ and $a\in\{1,2,\ldots,p^{r}\}$.
It follows that
\begin{align*}
\frac{\Gamma_p(a+tp)}{\Gamma_p(a)}&={\Gamma_p(tp)}\prod_{\substack{1\leq k<a\\ p\nmid k}}\frac{k+tp}{k}=
{\Gamma_p(tp)}\prod_{\substack{1\leq k<a\\ p\nmid k}}\bigg(1+\frac{tp}{k}\bigg)\\
&\equiv {\Gamma_p(tp)}\cdot\bigg(1+\sum_{j=1}^{r-1}\fH_{a-1}^{(j)}(p)\cdot(tp)^j\bigg)\pmod{p^{r}}.
\end{align*}
Therefore, by \eqref{Gammapalphatppr},
\begin{align*}
\frac{\Gamma_p(\alpha+tp)}{\Gamma_p(\alpha)}\equiv&\bigg(\sum_{k=0}^{r-1}\frac{G_k(0)}{k!}\cdot(tp)^k\bigg)\cdot\bigg(1+\sum_{j=1}^{r-1}\fH_{a-1}^{(j)}(p)\cdot(tp)^j\bigg)\\
\equiv&1+\sum_{k=1}^{r-1}(tp)^k\sum_{j=0}^{k}\frac{G_j(0)}{j!}\cdot \fH_{a-1}^{(k-j)}(p)
\pmod{p^{r-\sigma_r}}.
\end{align*}
By Lemma \ref{fHnmspl},
$$
\fH_{a-1}^{(k-j)}(p)\equiv
\fH_{\alpha-1}^{(k-j)}(p)\pmod{p^{r-\eta_{k-j}}}.
$$
In view of \eqref{nupGkx} and \eqref{Gammapalphatppr}, we have
\begin{equation}\label{Gkalphapr}
G_k(\alpha)\equiv
k!\sum_{j=0}^{k}\frac{G_j(0)}{j!}\cdot \fH_{a-1}^{(k-j)}(p)\pmod{p^{r-\max\{\sigma_r,\omega_r\}}},
\end{equation}
where $\sigma_r$ is the one defined in Corollary \ref{Gammapalphatpprc}.
$$
\omega_r:=\max_{0\leq j\leq r-1}\bigg\{\eta_{r-1-j}+\nu_p(j!)+\bigg\lfloor\frac jp\bigg\rfloor\bigg\}.
$$
Letting $r\to\infty$ in \eqref{Gkalphapr}, \eqref{Gkalpha} is derived.

Furthermore,
noting that $\sigma_r=\eta_r=\omega_r=0$ if $1\leq r\leq p-2$, we get \eqref{Gkalpha2}.
\end{proof}

\end{document}